\RequirePackage{fix-cm}
\documentclass[twocolumn]{svjour3_alt}          
\smartqed  

\usepackage{graphicx}

\usepackage[labelfont=bf]{caption} 

\usepackage{algorithm}
\usepackage{algpseudocode}
\usepackage{bm}
\usepackage{mathptmx}    
\usepackage{amsmath,amssymb}
\usepackage[utf8]{inputenc}

\usepackage{commath} 
\usepackage{stmaryrd} 
\usepackage{enumerate} 
\usepackage[numbers]{natbib} 
\usepackage{hyperref}
\usepackage[title]{appendix}

\usepackage{tikz}
\usetikzlibrary{spy,calc}
\usepackage{pgfplots}
\pgfplotsset{compat=1.13}
\usepackage{subcaption}

\hypersetup{
    colorlinks,
    citecolor=blue,
    filecolor=black,
    linkcolor=black,
    urlcolor=blue
}

\renewcommand{\secref}{Section~\ref}

\newcommand{\inner}[1]{\langle #1 \rangle}
\newcommand{\biginner}[1]{\left\langle #1 \right\rangle}
\newcommand{\eps}{\varepsilon}

\newcommand{\NN}{\mathbb{N}} 
 \newcommand{\RR}{\mathbb{R}}

\newcommand{\calC}{\mathcal{C}}

\DeclareMathOperator{\dist}{dist}
\DeclareMathOperator{\dom}{dom}

\DeclareMathOperator{\sgn}{sgn}

\DeclareMathOperator{\symm}{symm}

\DeclareMathOperator*{\argmin}{arg\,min}

\numberwithin{equation}{section}
\numberwithin{figure}{section}

\newtheorem{thm}{Theorem}
\numberwithin{thm}{section}

\newtheorem{defn}[thm]{Definition}
\newtheorem{lem}[thm]{Lemma}
\newtheorem{prop}[thm]{Proposition}
\newtheorem{ass}[thm]{Assumption}
\newtheorem{rem}[thm]{Remark}
\newtheorem{ex}[thm]{Example}
\renewenvironment{proof}{\emph{Proof.}}{\qed}

\DeclareMathAlphabet{\mathcal}{OMS}{cmsy}{m}{n} 

\journalname{Journal of Mathematical Imaging and Vision}

\begin{document}

\title{Bregman Itoh--Abe methods for sparse optimisation \thanks{This article is dedicated to Mila Nikolova whose keen interest and inspiring discussions have encouraged our research on geometric integration for optimisation. MB acknowledges support from the Leverhulme Trust Early Career Fellowship ECF-2016-611 `Learning from mistakes: a supervised feedback-loop for imaging applications’. ESR and CBS acknowledge support from CHiPS (Horizon 2020 RISE project grant) and from the Cantab Capital Institute for the Mathematics of Information. CBS acknowledges support from the Leverhulme Trust project on ``Breaking the non-convexity barrier”, EPSRC grant Nr EP/M00483X/1, and the EPSRC Centre Nr EP/N014588/1.  Moreover, CBS acknowledges support from the RISE project NoMADS and the Alan Turing Institute.}}
 
\titlerunning{Bregman Itoh--Abe methods for sparse optimisation}        

\author{Martin Benning  \and
        Erlend S. Riis \and
        Carola-Bibiane Sch{\"o}nlieb
}


\institute{Martin Benning \at
              School of Mathematical Sciences, Queen Mary University of London, UK \\
              \email{m.benning@qmul.ac.uk }           
           \and
           Erlend S. Riis \at
              Department of Applied Mathematics and Theoretical Physics, University of Cambridge, UK \\
              \email{e.s.riis@damtp.cam.ac.uk} 
		  \and
           Carola-Bibiane Sch{\"o}nlieb \at
              Department of Applied Mathematics and Theoretical Physics, University of Cambridge, UK \\
              \email{cbs31@cam.ac.uk} 
}

\date{}

\maketitle

\begin{abstract}
In this paper we propose optimisation methods for variational regularisation problems based on discretising the inverse scale space flow with discrete gradient methods. Inverse scale space flow generalises gradient flows by incorporating a generalised Bregman distance as the underlying metric. Its discrete-time counterparts, Bregman iterations and linearised Bregman iterations, are popular regularisation schemes for inverse problems that incorporate a priori information without loss of contrast. Discrete gradient methods are tools from geometric numerical integration for preserving energy dissipation of dissipative differential systems. The resultant Bregman discrete gradient methods are unconditionally dissipative, and achieve rapid convergence rates by exploiting structures of the problem such as sparsity. Building on previous work on discrete gradients for non-smooth, non-convex optimisation, we prove convergence guarantees for these methods in a Clarke subdifferential framework. Numerical results for convex and non-convex examples are presented.

\noindent\textbf{Keywords.} \hspace{0.1cm} Non-convex optimisation, non-smooth optimisation, Bregman iteration, inverse scale space, geometric numerical integration, discrete gradient methods

\noindent\textbf{AMS subject classifications.} \hspace{0.1cm} 49M37, 49Q15, 65K10, 90C26
\end{abstract}

\section{Introduction}
\label{intro}

We consider the constrained optimisation problem
\begin{equation}
\label{eq:main}
\min_{\vec{x} \in \calC} V(\vec{x}),
\end{equation}
for an objective function \(V:\RR^n \to \RR\) and constraint \(\calC \subset \RR^n\). The function \(V\) may be non-convex and non-smooth, as outlined in \assref{ass:JV}. In this paper, we propose and study optimisation schemes by using tools from geometric numerical integration to solve the \emph{inverse scale space} (ISS) flow.

The ISS flow is a differential system which generalises gradient flows by replacing the Euclidean distance by a \emph{Bregman distance}, defined via a convex \emph{Bregman (distance generating) function} \(J: \RR^n \to \overline{\RR}\). The ISS flow is given by
\begin{equation}
\label{eq:ISS}
\begin{aligned}
\dot{\vec{p}}(t) = -\nabla V(\vec{x}(t)), \quad \vec{p}(t) \in \partial J(\vec{x}(t)).
\end{aligned}
\end{equation}
Here, \(\partial J\) is the convex subdifferential of \(J\), to be defined in \secref{sec:preliminaries}, and \(\overline{\RR} := \RR \cup \{\pm \infty\}\). The term \emph{inverse scale space flow} goes back to Scherzer \& Groetsch \citep{sch01}. This flow is typically derived as the continous-time limit of \emph{Bregman iterations}---methods for solving variational regularisation problems. Like gradient flows, the ISS flow is a dissipative system, and its dissipative structure is determined by the function \(J\). This allows one to solve \eqref{eq:main} while incorporating \emph{a priori} information into the optimisation scheme, with the benefits of converging to superior solutions, and doing so faster. For background on the ISS flow and Bregman distances, see \secref{sec:convex}.

Geometric numerical integration deals with numerical integration methods that preserve geometric structures of the continuous system. Such geometric structures include dissipation or conservation of energy, and Lyapunov functions. In recent years, geometric numerical integration---and numerical integration in general---has gained interest within the mathematical optimisation community, as a framework for formulating iterative schemes that are dissipative or amenable to Lyapunov arguments.

We propose to discretise the inverse scale space flow with discrete gradient methods. These are methods from geometric numerical integration that preserve the aforementioned geometric structures in a general setting. In recent papers \citep{rii18smooth,gri17,rii18nonsmooth,rin18}, optimisation schemes based on discretising gradient flows with discrete gradients have been analysed and implemented for various problems. Favourable properties of the discrete gradient methods include unconditional dissipation, i.e. dissipation is ensured for any time step, and the Itoh--Abe discrete gradient method is derivative-free, and has convergence guarantees in the non-smooth, non-convex setting \citep{rii18nonsmooth}. For smooth problems, the theoretical convergence rates of the discrete gradient methods match those of explicit gradient descent and coordinate descent \citep{rii18smooth}. The drawback of these methods is that the updates are in general implicit. Nevertheless, for many simple variational problems, the updates turn out to be explicit.

In this paper, we study the Itoh--Abe discrete gradient method applied to the ISS flow. We prove that the method is well-defined and converges to a set of stationary points for non-smooth, non-convex functions. Furthermore, building on the paper by Miyatake et al. \citep{miy18} which pointed out the equivalence between the discrete gradient methods for linear systems and successive-over-relaxation (SOR) methods, we point out equivalencies of various approaches to least squares problems.

Bregman iterations, and related methods, are closely tied to inverse problems and regularisation methods, particularly in signal processing. We consider numerical examples in this setting as well.

\subsection{Related literature}

Spurred by applications for variational regularisation in image processing and compressed sensing, the ISS flow and the Bregman method have been active areas of research during the last decade. The Bregman iterative method was originally proposed by Osher et al. \citep{osh05} in 2005 for total variation-based image denoising, representing an extension of the Bregman proximal algorithm \citep{cen92, eck93, kiw97, teb92} to non-smooth Bregman functions. Subsequently the ISS flow was derived and analysed by Burger et al. \citep{bur06, bur07, bur13iss, bur16bregman}. Since then, researchers have studied the ISS flow with applications to generalised spectral analysis in a nonlinear setting, i.e. by Burger et al. \citep{bur16}, Gilboa et al. \citep{gil16}, and Schmidt et al. \citep{sch18}. The Bregman method has been studied for \(\ell^1\)-regularisation and compressed sensing by Goldstein \& Osher \citep{gol09} and Yin et al. \citep{yin08}, and extended to primal-dual algorithms by Zhang et al. \citep{zha11}.

The linearised Bregman method was proposed by Yin et al. \citep{yin08} for applications to \(\ell^1\)-regularisation and compressed sensing, and further studied in this setting by Cai et al. \citep{cai09}, and Dong et al. \citep{don10}. An extension for non-convex problems was proposed by Benning et al. \citep{ben17}, proving global convergence for functions that satisfy the Kurdyka--{\L}ojasiewicz property. Lorenz et al. \citep{lor14, sch18kaczmarz} proposed a sparse variant of the Kaczmarz method for linear problems based on linearised Bregman iterations. These and other methods were unified in a Split Feasibility Problems framework for genereal convergence results by Lorenz et al. \citep{lor14sfp}. For further details on Bregman iterations and linearised Bregman methods, we refer to \citep{ben18}.

We review papers on discrete gradient methods for optimisation based on gradient flows. Grimm et al. \citep{gri17} used discrete gradients to solve variational regularisation problems in image analysis, and proved convergence to a set of stationary points for continuously differentiable functions. Ehrhardt et al. \citep{rii18smooth} provided additional analysis for the methods, including convergence rates for smooth, convex problems and Polyak--{\L}ojasiewicz functions, as well as well-posedness of the implicit equation. Ringholm et al. \citep{rin18} applied the Itoh--Abe discrete gradient method to non-convex image problems with Euler's elastica regularisation.

Furthermore, several recent works have looked at discrete gradient methods in more general settings. Riis et al. \citep{rii18nonsmooth} studied the Itoh--Abe discrete gradient method in the setting of derivative-free optimisation of non-smooth, non-convex objective functions, and proved that the method converges to a set of stationary points in the Clarke subdifferential framework. Celledoni et al. \citep{cel18} extended the Itoh--Abe discrete gradient method to optimisation problems defined on Riemannian manifolds. Hern{\'a}ndez-Solano et al. \citep{her15} combined a discrete gradient method with Hopfield networks in order to preserve a Lyapunov function for optimisation problems.

\textcolor{black}{We point out that a central feature of discrete gradient methods is that due to their implicit formulation, no restrictions are required for the time steps. This will also be the case for the analysis in this paper.}

Beyond discrete gradients, other methods from numerical integration have led to notable developments in optimisation in recent years. Recent papers by Su et al. \citep{su16} and Wibisono et al. \citep{wib16} study second-order ODEs based on the continuous-time limit of Nesterov's accelerated gradient descent \citep{nes83}, in order to gain insight into the acceleration phenomenon. In this setting, Wilson et al. \citep{wil16} provided a framework for Lyapunov analysis of optimisation schemes and continuous-time dynamics, and Betancourt et al. \citep{bet18} proposed a framework for symplectic integration for optimisation. On a related note, Scieur et al. \citep{sci17} showed that several accelerated optimisation schemes can be derived as multi-step integration schemes from numerical analysis. An optimisation scheme based on numerically integrating dissipative Hamiltonian conformal systems was proposed by Maddison et al. \citep{mad18}, with the aim of ensuring linear convergence for a larger group of functions than classical methods. Other numerical integration methods, such as implicit Runge-Kutta methods, where energy dissipation is ensured under mild time step restrictions \citep{hai13}, and explicit stabilised methods for solving strongly convex problems \citep{eft18}. Another example is the study of gradient flows in metric spaces and minimising movement schemes \citep{amb08}, concerning gradient flow trajectories under other measures of distance, such as the Wasserstein metric \citep{san17}.

\subsection{Structure and contributions}

The paper is structured as follows. For the remainder of Section 1, we define mathematical notation. In Section 2, we provide preliminary material for convex and non-convex analysis, and introduce Bregman distances. In Section 3, we propose a Bregman discrete gradient method based on the ISS flow, and prove well-posedness and convergence results in a non-convex, non-smooth framework. In Section 4, we discuss particular examples of Bregman discrete gradient methods. In Section 5, we present results from numerical experiments. In Section 6, we conclude.

\subsection{Notation and preliminaries}

Throughout the paper, we denote by \(\|\cdot\|\) the \(\ell^2\)-norm, i.e. \(\|\vec{x}\|^2 = \sum_{i=1}^n |x_i|^2\). For \(p \in [1,+\infty]\), \(\|\cdot\|_p\) denotes the usual \(\ell^p\)-norm. For \(\eps > 0\) and \(\vec{x} \in \RR^n\), we denote by \(B_\eps(\vec{x})\) the open ball \(\{\vec{y} \in \RR^n \; : \; \|\vec{y} - \vec{x}\| < \eps \}\).
We denote by \(\{\vec{e}^1, \ldots, \vec{e}^n\}\) the standard coordinate vectors in \(\RR^n\). We define the extended reals as \(\overline{\RR} := \RR \cup \{ \pm \infty\}\).

For a sequence \((\vec{x}^k)_{k\in\NN} \subset \RR^n\), we denote by \(S\) its limit set, i.e. the set of accumulation points,
\[
S = \set{\vec{x}^* \in \RR^n \; : \; \exists (\vec{x}^{k_j})_{j \in \NN} \; \mbox{ s.t. } \; \vec{x}^{k_j} \to \vec{x}^*}.
\]
For a matrix \(A \in \RR^{m\times n}\), we denote by \(\vec{a}^i \in \RR^n\) its \(i\)th row, and \(A^T\) its transpose. \textcolor{black}{Accordingly, \(a^i_j\) refers to the \(j\)th element of the \(i\)th column of \(A\).}

\textcolor{black}{For \(x \in \RR\), the \emph{sign operator} \(\sgn: \RR \to \{\pm 1, 0\}\) is defined as
\[
\sgn(x) = \begin{cases}
1, \quad &\mbox{if } x > 0, \\
0, &\mbox{if } x = 0, \\
-1, &\mbox{if } x < 0.
\end{cases}
\]}

\section{Preliminaries for convex and non-convex analysis}
\label{sec:preliminaries}

In this section, we provide preliminary material on nonsmooth analysis. In \secref{sec:convex}, we cover convex analysis and Bregman distances, while in \secref{sec:nonconvex} we cover nonconvex, nonsmooth analysis in the Clarke subdifferential framework.

\subsection{Convex analysis}

\label{sec:convex}

We consider functions \(J:\RR^n \to \overline{\RR}\) that are convex, proper, and lower-semicontinuous (see \citep{eke99,roc15} for details on this class of functions).

\begin{defn}[Effective domain]
The \emph{effective domain} of a function \(J: \RR^n \to \overline{\RR}\) is defined as \(\dom(J) = \{\vec{x} \in \RR^n \; : \; J(\vec{x}) < \infty \}\).
\end{defn}

\begin{defn}[Subgradients and subdifferentials]
The \emph{subdifferential} of a convex function \(J:\RR^n \to \overline{\RR}\) at \(\vec{x} \in \RR^n\) is the set of vectors
\[
\partial J(\vec{x}) = \{ \vec{p} \in \RR^n \; : \; J(\vec{y}) - J(\vec{x}) - \inner{\vec{p}, \vec{y}-\vec{x}} \geq 0 \mbox{ for all } \vec{y} \in \RR^n\}.
\]
Vectors in \(\partial J(\vec{x})\) are called \emph{subgradients} of \(J\) at \(\vec{x}\).
\end{defn}
For \(i=1,\ldots,n\), we denote by \([\partial J(\vec{x})]_i\) the projection of \(\partial J(\vec{x})\) onto the \(i\)th coordinate, i.e. \([\partial J(\vec{x})]_i = \{p_i \; : \; \vec{p} \in \partial J(\vec{x})\}\).

For the theory of convex functions and their subdifferentials, see \citep{roc15}.

We consider the \emph{characteristic function} \(\chi_{\calC}\) of a convex, closed set \(\calC \subset \RR^n\), defined by
\begin{equation*}
\chi_{\calC}(\vec{x}) := \begin{cases}
0, \quad &\mbox{if } \vec{x} \in \calC, \\
+\infty, &\mbox{else.}
\end{cases}
\end{equation*}
This is a convex, proper, lower-semicontinuous function, and \(\vec{0} \in \partial \chi_{\calC}(\vec{x})\) for all \(\vec{x} \in \dom(\chi_{\calC}) = \calC\).

We can now define the generalised Bregman distance \citep{bre67} of a convex function \(J\).
\begin{defn}[Bregman distance]
Given \(\vec{p} \in \partial J(\vec{x})\), the \emph{Bregman distance between \(\vec{x}\) and \(\vec{y}\)} is given by
\[
D^p_J(\vec{x},\vec{y}) = J(\vec{y}) - J(\vec{x}) - \inner{\vec{p}, \vec{y}-\vec{x}}.
\]
We refer to \(J\) as the corresponding \emph{Bregman function}.
\end{defn}
\begin{ex}
\label{ex:euclidean}
If \(J(\vec{x}) = \|\vec{x}\|^2/2\), then \(D^p_J(\vec{x},\vec{y}) = \|\vec{x}-\vec{y}\|^2/2\), i.e. the square of the Euclidean distance.
\end{ex}

While Bregman distances are non-negative due to convexity of \(J\), they are not metrics as they do not generally satisfy symmetry or a triangle inequality. However, a related object is the symmetric Bregman distance.
\begin{defn}[Symmetric Bregman distance]
Given \(\vec{p} \in \partial J(\vec{x})\) and \(\vec{q} \in \partial J(\vec{y})\), the \emph{symmetric Bregman distance between \(\vec{x}\) and \(\vec{y}\)} is given by
\[
D^{\symm}_J(\vec{x},\vec{y}) = D^q_J(\vec{y},\vec{x}) + D_J^p(\vec{x},\vec{y}) = \inner{\vec{q}-\vec{p},\vec{y}-\vec{x}}.
\]
\end{defn}

\begin{defn}[\(\mu\)-convexity]
\label{defn:strongly_convex}
A proper, lower-semicontinuous convex function \(J: \RR^n \to \overline{\RR}\) is \emph{\(\mu\)-convex} for \(\mu \geq 0\) if either of the following (equivalent) conditions hold.
\begin{enumerate}[(i)]
	\item The function \(J(\cdot) - \dfrac{\mu}{2} \|\cdot\|^2\) is convex.
	\item \(J(\vec{y}) - J(\vec{x}) - \inner{\vec{p},\vec{y}-\vec{x}} \geq \tfrac{\mu}{2} \|\vec{y}-\vec{x}\|^2\) for all \(\vec{x},\vec{y} \in \RR^n\), \(\vec{p} \in \partial J(\vec{x})\).
\end{enumerate}
If \(\mu > 0\), \(J\) is said to be \emph{strongly convex}.
\end{defn}
For strongly convex Bregman functions, the following property of Bregman distances is immediate.
\begin{prop}[Bregman distance under \(\mu\)-convexity]\hfill
If \(J\) is \(\mu\)-convex, then for all \(\vec{x}, \vec{y} \in \RR^n\),
\[
D_J^p(\vec{x},\vec{y}) \geq \frac{\mu}{2} \|\vec{y}-\vec{x}\|^2, \quad D_J^{\symm}(\vec{x},\vec{y}) \geq \mu\|\vec{y}-\vec{x}\|^2.
\]
\end{prop}

\subsection{Non-convex subdifferential analysis}
\label{sec:nonconvex}

We summarise the main concepts of the Clarke subdifferential for locally Lipschitz continuous, non-smooth, non-convex functions \(V: \RR^n \to \RR\), and refer to \citep{cla90} for further details.

\begin{defn}[Lipschitz continuity]
\(V\) is \emph{Lipschitz of rank \(L\) near \(\vec{x}\)} if there exists \(\eps > 0\) such that for all \(\vec{y}, \vec{z} \in B_\eps(\vec{x})\), one has
\[
|V(\vec{y}) - V(\vec{z})| \leq L\|\vec{y}-\vec{z}\|.
\]
\(V\) is locally Lipschitz continuous if the above property holds for all \(\vec{x} \in \RR^n\).
\end{defn}

\begin{defn}[Clarke directional derivative]
\label{defn:clarke_derivative}
For \(V\) Lipschitz near \(\vec{x}\) and for a vector \(\vec{d} \in \RR^n\),  the \emph{Clarke directional derivative} is given by
\[
V^o(\vec{x} ; \vec{d}) = \limsup_{\vec{y} \to \vec{x}, \; \lambda \downarrow 0} \dfrac{V(\vec{y}+\lambda \vec{d}) - V(\vec{y})}{\lambda}.
\]
\end{defn}
\begin{defn}[Clarke subdifferential]
\label{defn:subdifferential}
Let \(V\) be locally Lipschitz continuous and \(\vec{x} \in \RR^n\). The \emph{Clarke subdifferential} of \(V\) at \(\vec{x}\) is given by
\[
\partial V(\vec{x}) = \set{\vec{p} \in \RR^n \; : \; V^o(\vec{x}; \vec{d}) \geq \inner{\vec{d}, \vec{p}} \mbox{ for all } \vec{d} \in \RR^n}.
\]
An element of \(\partial V(\vec{x})\) is called a \emph{Clarke subgradient}.
\end{defn}
The Clarke subdifferential was introduced by Clarke in \citep{cla73}. It is well-defined for locally Lipschitz functions, coincides with the standard subdifferential for convex functions \citep[Proposition 2.2.7]{cla90}, and coincides with the derivative at points of strict differentiability \citep[Proposition 2.2.4]{cla90}. We additionally state two useful results, both of which can be found in Chapter 2 of \citep{cla90}.
\begin{prop}[Properties of Clarke subdifferential] \hfill
\label{prop:subdiff_properties}
Suppose \(V\) is locally Lipschitz continuous. Then
\begin{enumerate}[(i)]
	\item \(\partial V(\vec{x})\) is nonempty, convex and compact, and if \(V\) is Lipschitz of rank \(L\) near \(\vec{x}\), then \(\partial V(\vec{x}) \subseteq B_L(\vec{0})\).
	\item \(\partial V(\vec{x})\) is outer semicontinuous. That is, for all \(\eps > 0\), there exists \(\delta > 0\) such that
	\[
	\partial V(\vec{y}) \subset \partial V(\vec{x}) + B_\eps(\vec{0}), \quad \mbox{ for all } \vec{y} \in B_\delta(\vec{x}).
	\]
\end{enumerate}
\end{prop}

In this paper, we will consider constrained optimisation problems \(\argmin_{\vec{x} \in \calC} V(\vec{x})\) for a closed, convex subset \(\calC\). We assume in the convergence analysis that \(\calC\) is box constraints of the form \(\calC = \otimes_{i=1}^n [l_i, u_l]\). To have a formal notion of first-order optimality in this setting, we define the \emph{tangent cone}.
\begin{defn}[Tangent cone]
For a convex, closed set \(\calC \subset \RR^n\), the \emph{tangent cone at \(\vec{x} \in \calC\)}, denoted by \(T(\calC,\vec{x})\), is the closure of the set of vectors \(\vec{v} \in \RR^n\) such that there is \(\delta > 0\) such that for all \(\eps \in (0,\delta)\), \(\vec{x} + \eps \vec{v} \in \calC\).
\end{defn}

\begin{defn}[Clarke stationarity]
Let \(V: \RR^n \to \RR\) be a locally Lipschitz continuous function and \(\calC \subset \RR^n\) a closed, convex set. A point \(\vec{x}^* \in \RR^n\) is a \emph{Clarke stationary point of \(V\) restricted to \(\calC\)} if \(V^o(\vec{x}^*; \vec{d}) \geq 0\) for all \(\vec{d} \in T(\calC,\vec{x}^*)\).
\end{defn}
As the following proposition shows, the above definition generalises minimisers for convex functions, and Clarke stationary points in the unrestricted case.
\begin{prop}[Stationary versus minimal]
Let \(V:\RR^n \to \RR\) be a convex function and \(\calC \subset \RR^n\) be closed and convex. Then \(\vec{x}^* \in \RR^n\) is a Clarke stationary point restricted to \(\calC\) if and only if \(\vec{x}^* \in \argmin_{\vec{x} \in \calC} V(\vec{x})\).
\end{prop}
\begin{proof}
This follows directly from \citep[Theorem 4.19]{jah07}, noting that for convex functions the Clarke directional derivative coincides with the classical directional derivative \citep[Proposition 2.2.7]{cla90}.
\end{proof}

\section{The discrete gradient method for the inverse scale space flow}

In what follows, we discuss the inverse scale space flow, and propose a \emph{Bregman discrete gradient method} by using discrete gradients to solve the differential system.

\subsection{Inverse scale space flow and Bregman methods}

For a convex function \(J: \RR^n \to \overline{\RR}\), objective function \(V: \RR^n \to \RR\) and starting points \(\vec{x}(0) = \vec{x}^0 \in \RR^n\), \(\vec{p}(0) \in \partial J(\vec{x}^0)\), the ISS flow is the dissipative differential system given by \eqref{eq:ISS}. If \(J\) were twice continuously differentiable and \(\mu\)-convex, then \eqref{eq:ISS} could be rewritten as
\[
\dot{\vec{x}}(t) = - (\nabla^2 J(\vec{x}(t)))^{-1} \nabla V(\vec{x}(t)),
\]
and the energy \(V(\vec{x}(t))\) would dissipate over time as
\begin{align*}
\dod{}{t}V(\vec{x}(t)) &= \biginner{\dot{\vec{x}}(t),\nabla V(\vec{x}(t))} \\
&= - \biginner{\dot{\vec{x}}(t), \nabla^2 J(\vec{x}(t))\dot{\vec{x}}(t) } \leq -\mu \|\dot{\vec{x}}(t)\|^2.
\end{align*}

We briefly discuss variants of Bregman methods as discretisations of \eqref{eq:ISS}. The Bregman method is derived by backward Euler discretisation of \eqref{eq:ISS}, and is given by
\begin{equation*}
\vec{p}^{k+1} = \vec{p}^k - \tau_k \nabla V(\vec{x}^{k+1}), \quad \vec{p}^{k+1} \in \partial J(\vec{x}^{k+1})
\end{equation*}
which can be rewritten as
\begin{equation}
\label{eq:bregman_method}
\vec{x}^{k+1} = \argmin_{\vec{x}\in\RR^n} V(\vec{x}) + \frac{1}{\tau_k} D_J^{\vec{p}^k}(\vec{x}^k, \vec{x}).
\end{equation}
From \eqref{eq:bregman_method}, we see that the Bregman method is dissipative, as
\begin{align*}
V(\vec{x}^{k+1}) - V(\vec{x}^k) &\leq - \frac{1}{\tau_k} D_J^{\vec{p}^k}(\vec{x}^k,\vec{x}^{k+1}) \\
&\leq - \frac{\mu}{2 \tau_k} \|\vec{x}^k - \vec{x}^{k+1}\|^2.
\end{align*}

Similarly, the linearised Bregman method is derived by forward Euler discretisation of \eqref{eq:ISS}, and is given by
\begin{equation*}
\vec{p}^{k+1} = \vec{p}^k - \tau_k \nabla V(\vec{x}^{k}), \quad \vec{p}^{k+1} \in \partial J(\vec{x}^{k+1})
\end{equation*}
or equivalently
\begin{equation*}
\vec{x}^{k+1} = \argmin_{\vec{x}\in\RR^n} V(\vec{x}^k) + \inner{\nabla V(\vec{x}^k), \vec{x} - \vec{x}^k} + \frac{1}{\tau_k} D_J^{\vec{p}^k}(\vec{x}^k, \vec{x}).
\end{equation*}

The ISS flow and Bregman methods are considered for solving ill-conditioned linear systems \(A\vec{x} = \vec{b}\), with the objective function
\[
V(\vec{x}) = \frac{1}{2}\|A\vec{x} - \vec{b}\|^2.
\]
In this setting, iterates of both the Bregman method and the linearised Bregman method converge \citep{ben18, lor14sfp} to a solution of 
\[
\min_{\vec{x}\in\RR^n} \{J(\vec{x}) \mbox{ s.t. } A\vec{x} = \vec{b}\}.
\]
Furthermore, applications of the ISS flow include image denoising with reduced contrast-loss and staircasing effects \citep{osh05}, recovering eigenfunctions \citep{sch18}, and identifying sparse or low-rank structures \citep{yin08}.

We make the following assumptions on the the objective function \(V\), the constraints \(\calC\), as well as the Bregman function \(J\).
\begin{ass}
\label{ass:JV}\hfill
\begin{enumerate}[a)]
	\item The function \(V: \RR^n \to \RR\) is locally Lipschitz continuous and bounded below.
	\item \(\vec{x}^* \in \RR^n\) is a Clarke stationary point of \(V\) restricted to \(\calC\) if and only if for all coordinate vectors \(\vec{e}^i\), we have \(V^o(\vec{x}^*; \vec{e}^i), V^o(\vec{x}^*; -\vec{e}^i) \geq 0\).
	\item The set \(\calC \subset \RR^n\)	consists of coordinate-wise box constraints, i.e. \(\calC = \otimes_{i=1}^n [l_i,u_i]\).
	\item The function \(J:\RR^n\to \overline{\RR}\) is proper, lower-semicontinuous, and \(\mu\)-convex with \(\mu > 0\). Furthermore, \(J(\vec{x}) = \sum_{i=1}^n j_i(x_i) + \chi_{[l_i,u_i]}(x_i)\), where \(j_i: \RR \to \overline{\RR}\) are convex and Lipschitz continuous on \([l_i,u_i]\) for each \(i\).
\end{enumerate}
\end{ass}

\subsection{The Bregman discrete gradient method}

In what follows, we define discrete gradients, and propose the Bregman discrete gradient method by discretising the ISS flow.
\begin{defn}[Discrete gradient]
\label{defn:dg}
Let \(V\) be a continuously differentiable function. A \emph{discrete gradient} is a continuous map \(\overline{\nabla} V: \RR^n \times \RR^n \to \RR^n\) such that for all \(\vec{x}, \vec{y} \in \RR^n\),
\begin{alignat}{2}
\label{eq:mean_value}
\inner{\overline{\nabla} V(\vec{x},\vec{y}), \vec{y}-\vec{x}} &= V(\vec{y}) - V(\vec{x}) & \quad &\mbox{(Mean value)}, \\
\label{eq:consistency}\lim_{\vec{y} \to \vec{x} }\overline{\nabla} V(\vec{x},\vec{y}) &= \nabla V(\vec{x}) & \quad &\mbox{(Consistency)}.
\end{alignat}
\end{defn}

Discrete gradients are tools from geometric numerical integration. In geometric integration, one studies methods for numerically solving ODEs while preserving certain structures of the continuous system---see \citep{hai06, mcl01} for an introduction. Discrete gradients are tools for solving first-order ODEs that preserve energy conservation laws, dissipation laws, and Lyapunov functions \citep{gon96, ito88, mcl99, qui96}.

\emph{The Itoh--Abe discrete gradient} \citep{ito88} (also known as the coordinate increment discrete gradient) is given by
\begin{equation}
\label{eq:IA}
\overline{\nabla} V(\vec{x},\vec{y}) = \begin{pmatrix}
\frac{V\del{y_1, x_2, \ldots, x_n} - V(\vec{x})}{y_1 - x_1}  \\
\frac{V\del{y_1, y_2, x_3, \ldots, x_n} - V\del{y_1, x_2, \ldots, x_n}}{y_2 - x_2} \\
\vdots \\
\frac{V(\vec{y}) - V\del{y_1, \ldots, y_{n-1}, x_n}}{y_n - x_n}
\end{pmatrix},
\end{equation}
where \(0/0\) is interpreted as \(\partial_i V(\vec{x})\).

The Itoh--Abe discrete gradient is derivative-free, and is evaluated by successively computing difference quotients. For optimisation problems where the gradient is expensive to compute (e.g. \citep{rin18}), or unavailable (e.g. \citep{rii18nonsmooth}), Itoh--Abe discrete gradient methods are useful for ensuring convergence guarantees and competitive convergence rates \citep{rii18smooth} without requiring derivatives.

We propose the \emph{Bregman discrete gradient method} as follows. For a starting point \(\vec{x}^0\in\RR^n\), subgradient \(\vec{p}^0 \in \partial J(\vec{x}^0)\), and time steps \((\tau_k)_{i \in \NN} > 0\), solve for \(k = 0,1,\ldots\),
\begin{equation}
\label{eq:bdg}
\begin{aligned}
\vec{p}^{k+1} &= \vec{p}^k - \tau_k \overline{\nabla} V(\vec{x}^k, \vec{x}^{k+1})  \\
\vec{p}^{k+1} &\in \partial J(\vec{x}^{k+1}).
\end{aligned}
\end{equation}
This scheme preserves the dissipative structure of the ISS flow and (linearised) Bregman methods, as we see by applying the mean value property \eqref{eq:mean_value} and \eqref{eq:bdg}.
\begin{align}
V(\vec{x}^k) - V(\vec{x}^{k+1}) &= \inner{\vec{x}^k - \vec{x}^{k+1}, \overline{\nabla} V(\vec{x}^k, \vec{x}^{k+1})} \nonumber \\
&=  \frac{1}{\tau_k} \inner{\vec{x}^k - \vec{x}^{k+1}, \vec{p}^k - \vec{p}^{k+1} } \nonumber \\
&= \frac{1}{\tau_k} D^{\symm}_{J}(\vec{x}^k, \vec{x}^{k+1}) \nonumber \\
&\geq \frac{\mu}{\tau_k} \|\vec{x}^k - \vec{x}^{k+1}\|^2. \label{eq:bdg_dissipation}
\end{align}
Furthermore, if we plug in \(J(\vec{x}) = \|\vec{x}\|^2/2\), we recover the discrete gradient method for gradient flows \citep{rii18smooth}. \textcolor{black}{Observe that the dissipative structure holds for arbitrary positive time steps \(\tau_k > 0\). We thus only require that \(\tau_k \in [\tau_{\min} , \tau_{\max}]\) for arbitrary bounds \(\tau_{\min}, \tau_{\max}  >0\).}

By \assref{ass:JV} d), the subdifferential of \(J\) is separable in the coordinates, i.e.
\[
\partial J(\vec{x}) = \prod_{i=1}^n \partial \chi_{[l_i,u_i]}(x_i) + \partial j_i(x_i).
\]
It follows that solving the Bregman Itoh--Abe equation \eqref{eq:bdg} corresponds to successively solving \(n\) scalar inclusions, i.e. given \(\vec{x}^k\), and \(\vec{p}^k\), we want \(\vec{x}^{k+1}\) and \(\vec{p}^{k+1}\) that solve
\textcolor{black}{
\begin{equation}
\label{eq:bia}
\begin{aligned}
p^{k+1}_i &= p^k_i - \tau_{k,i} \frac{V(\vec{y}^{k,i}) - V(\vec{y}^{k,i-1})}{x^{k+1}_i-x^k_i}, \\
p^{k+1}_i &\in \partial j_i(y^{k,i}_i) + \partial \chi_{[l_i,u_i]}(y^{k,i}_i), \\
\vec{y}^{k,i} &= [x^{k+1}_1,\ldots, x^{k+1}_i, x^k_{i+1},\ldots, x^k_n],\\
i &= 1,\ldots,n.
\end{aligned}
\end{equation}}
Here \(\vec{y}^{k,i}\) denotes \([x^{k+1}_1, \ldots, x^{k+1}_i, x^k_{i+1}, \ldots, x^k_n]^T\). For a choice of \(v^k_i \in [\partial V(\vec{y}^{k,i-1})]_i\), if \(p^k_i - \tau_{k,i} v^k_{i} \in [\partial J(\vec{y}^{k,i-1})]_i\), then we consider \(x^{k+1}_i = x^k_i\) and \(p^{k+1}_i = p^{k}_i - \tau_{k,i} v^k_{i}\) an admissible update.

\begin{rem}
While discrete gradients are not defined for non-smooth functions \(V\), the resultant Bregman discrete gradient method \eqref{eq:bia} is still well-defined, provided the function \(V\) has a well-defined Clarke subdifferential. This is the case e.g. if \(V\) is locally Lipschitz continuous \citep{cla90}. Furthermore, note that the properties used to derive the dissipative structure \eqref{eq:bdg_dissipation} holds in this setting.
\end{rem}

\section{Theoretical results}
\label{sec:convergence}

\textcolor{black}{In this section, we prove that a solution to the Bregman discrete gradient method \eqref{eq:bia} exists, and provide conditions in which the update is unique. Furthermore, we prove that the accumulation points of the iterates \((\vec{x}^k)_{k\in\NN}\) from the Bregman Itoh--Abe method are Clarke stationary.}

\subsection{Well-posedness and preliminary results}

\textcolor{black}{The Bregman Itoh--Abe scheme \eqref{eq:bia} is in general implicit. We therefore want to ensure that an update exists and is easy to compute. Furthermore, we want to know under what conditions the updates are unique. In what follows, we adress these questions. The corresponding analysis generalises that from the previous papers on discrete gradient methods for optimisation \citep{rii18smooth,rii18nonsmooth}.}

\textcolor{black}{We first show that an update (not necessarily unique) always exists.}
\begin{lem}[Existence]
\label{lem:existence}
For any \(\tau > 0\), \(\vec{x}^k \in \RR^n\), and \(\vec{p}^k \in \partial J(\vec{x}^k)\), there exists an update \((\vec{x}^{k+1}, \vec{p}^{k+1})\) that satisfies \eqref{eq:bia}.
\end{lem}

\begin{proof}
As \eqref{eq:bia} consists of successive scalar updates, it is sufficient to consider a scalar problem, \(v: \RR \to \RR\), \(j: \RR \to \overline{\RR}\). For \(x \in \RR\) and \(p \in \partial j(x)\) we either want \(y \neq x\) such that
\begin{equation}
\label{eq:scalar_bia}
p - \tau \frac{v(y)-v(x)}{y-x} \in \partial j(y),
\end{equation}
or \(y = x\) and \(p - \tau w \in \partial j(x)\), for some \(w \in \partial v(x)\).

If such a \(w\) exists, we are done. Otherwise, we have \(\min\{v^o(x;1), v^o(x;-1)\} < 0\) and may assume that \(v^o(x;1) < 0\). In this case, we will show that there exists \(y > x\) such that \eqref{eq:scalar_bia} holds.

Since \(p-\tau v^o(x;1) > p\) and \(p \in \partial j(x)\), we deduce that \(p - \tau v^o(x;1) > \partial j(x)\). By the outer semicontinuity of subdifferentials and definition of Clarke directional derivatives, there is \(\delta > 0\) such that
\[
p - \tau \frac{v(y) - v(x)}{y-x} > \partial j(y) \mbox{ for all } y \in (x, x+\delta).
\]
On the other hand, as \(v\) is bounded below,
\[y \mapsto (v(y) - v(x))/(y-x)
\]
is bounded below for all \(y \in [x+\delta, +\infty)\), while by \(\mu\)-convexity of \(j\), we have \(\partial j(y) \geq \partial j(x) + \mu (y-x)\) for all \(y \in [x+\delta,+\infty)\). Hence, there is \(r \gg 0\) such that
\[
p - \tau \frac{v(y) - v(x)}{y-x} < \partial j(y) \mbox{ for all } y \geq x + r.
\]
By continuity of \(v\), and by outer semicontinuity of subdifferentials, it follows that there exists \(y \in (x+\delta, x+r)\) that solves \eqref{eq:scalar_bia}. This concludes the proof.
\end{proof}

\textcolor{black}{Next we give conditions in which the update is unique. While this holds for any time step if the objective function is convex, it may fail for nonconvex functions, as discussed in \citep[Section 5]{rii18smooth}.
\begin{prop}[Uniqueness]
\label{prop:unique}
For \(\tau > 0\), \(\vec{x}^k \in \RR^n\), and \(\vec{p}^k \in \partial J(\vec{x}^k)\), the update \((\vec{x}^{k+1}, \vec{p}^{k+1})\) to \eqref{eq:bia} is unique in the following cases.
\begin{enumerate}[(i)]
	\item \(V\) is convex.
	\item \( V(\cdot) + \eta \|\cdot\|^2\) is convex for \(\eta > 0\), and \(\tau < \mu/\eta\).
\end{enumerate}
\end{prop}}
\begin{proof}
\textcolor{black}{
\emph{Case (i):} Suppose \(V\) is convex. The existence of a solution to \eqref{eq:bia} is guaranteed by \lemref{lem:existence}. To establish uniqueness, we argue as follows. An update \(\vec{y}^{k,i}\) must satisfy
\begin{equation}
\label{eq:unique1}
p^{k}_i - \tau \frac{V(\vec{y}^{k,i}) - V(\vec{y}^{k,i-1})}{x^{k+1}_i - x^{k}_i} \in [\partial J(\vec{y}^{k,i})]_i.
\end{equation}
The left-hand-side is non-increasing with respect to \(x^{k+1}_i\) due to the difference quotient term of the convex function \(V\), while the right-hand side is strictly increasing due to the strong convexity of \(J\). Hence there cannot be two distinct solutions for \(y^{k,i}_i\) to the scalar equation. This implies uniqueness of the update.}

\textcolor{black}{\emph{Case (ii):} Suppose \(V(\cdot) + \eta \|\cdot\|^2\) is convex and \(\tau < \mu/\eta\). We rewrite \eqref{eq:unique1} as
\begin{equation*}
\frac{p^{k}_i}{\tau} - \frac{V(\vec{y}^{k,i}) - V(\vec{y}^{k,i-1})}{x^{k+1}_i - x^{k}_i} - \eta x^{k+1}_i \in \frac{1}{\tau} [\partial J(\vec{y}^{k,i})]_i - \eta x^{k+1}_i.
\end{equation*}
Since \(\tau \eta < \mu\), the function \(J(\cdot) - \tau\eta \|\cdot\|^2/2\) is strongly convex, and therefore the right-hand side is strictly increasing with respect to \(x^{k+1}_i\). On the other side, as \(V(\cdot) + \eta \|\cdot\|^2\) is convex, the difference quotient term \(\tfrac{V(\vec{y}^{k,i}) - V(\vec{y}^{k,i-1})}{x^{k+1}_i - x^{k}_i} + \eta x^{k+1}_i\) is non-decreasing with respect to \(x^{k+1}_i\).  Therefore, by the same reasoning as in the first case, the update to the Bregman Itoh--Abe method must be unique.}
\end{proof}
\begin{rem}
If the update is stationary, i.e. \(x^{k+1}_i = x^k_i\), then the subgradient update \(p^{k+1}_i\) is unique only up to the choice of subderivative \(v_i \in [\partial V(\vec{y}^{k,i-1})]_i\).

\textcolor{black}{Note that while the convexity criteria in \propref{prop:unique} are global, they are often generalised to local properties such as lower \(C^2\)-smoothness \citep{roc85} and prox-regularity \citep{pol96}. However, an analysis of these properties in the context of Itoh--Abe discrete gradient methods is beyond the theoretical and practical scope of this paper. We also add that lack of uniqueness is not an issue for the implementation or computational tractability of the scheme.}
\end{rem}

\textcolor{black}{The following lemma summarises several useful properties of the iterates \((\vec{x}^k, \vec{p}^k)_{k\in\NN}\), and generalises Lemmas 3.3, 3.4 in \citep{rii18nonsmooth}.}
\begin{lem}[Properties of scheme]
\label{lem:main}
Let \(V\), \(J\), and \(\calC\) satisfy \assref{ass:JV}, and let \((\vec{x}^k,\vec{p}^k)_{k\in\NN}\) be iterates that solve \eqref{eq:bia} for time steps \((\tau_k)_{k\in\NN} \subset [\tau_{\min},\tau_{\max}]\). Then the following properties hold.
\begin{enumerate}[(i)]
	\item \(V(\vec{x}^{k+1}) \leq V(\vec{x}^k)\).
	\item \(\lim_{k\to\infty} \|\vec{x}^{k+1} - \vec{x}^k\| = 0\).
	\item If \(V\) is coercive, then there exists a convergent subsequence of \((\vec{x}^k, \vec{p}^k)_{k\in\NN}\).
	\item The set of limit points \(S\) is compact, connected, and has empty interior. Furthermore, \(V\) is single-valued on \(S\).
\end{enumerate}
\end{lem}

\begin{proof}
Property \emph{(i)} follows from \eqref{eq:bdg_dissipation}.

As \(V\) is bounded below and \((V(\vec{x}^k))_{k \in \NN}\) is decreasing, \(V(\vec{x}^k) \to V^*\) for some limit \(V^*\). Therefore, by \eqref{eq:bdg_dissipation},
\begin{align*}
V(\vec{x}^0) - V^* &= \sum_{k = 0}^\infty V(\vec{x}^k) - V(\vec{x}^{k+1}) \\
&\geq \sum_{k = 0}^\infty \dfrac{\mu}{\tau_k} \|\vec{x}^k - \vec{x}^{k+1}\|^2 \geq \frac{\mu}{\tau_{\max}} \sum_{k = 0}^\infty \|\vec{x}^k - \vec{x}^{k+1}\|^2.
\end{align*}
This implies property \emph{(ii)}.

Properties \emph{(iii)} and \emph{(iv)} follow from \emph{(i)} and \emph{(ii)} and are proven respectively in \citep[Lemma 3.3 and Lemma 3.4]{rii18nonsmooth}.
\end{proof}

\subsection{Amended scheme}

\textcolor{black}{In the next subsection, we prove that accumulation points of the iterates from the Bregman Itoh--Abe method are Clarke stationary. However, we first modify the Bregman Itoh--Abe method to `forget' subgradients induced by the constraints.}

\textcolor{black}{We define the modified scheme as follows. For a starting point \(\vec{x}^0\), \(\vec{p}^0 \in \partial J(\vec{x}^0)\), and \(k \in \NN\), update
\begin{equation}
\label{eq:bia2}
\begin{aligned}
p^{k+1}_i + \tilde{q}  &= p^k_i - \tau_{k,i} \frac{V(\vec{y}^{k,i}) - V(\vec{y}^{k,i-1})}{x^{k+1}_i-x^k_i}, \\
p^{k+1}_i &\in \partial j_i(y^{k,i}_i), \quad \tilde{q} \in \partial \chi_{[l_i,u_i]}(y^{k,i}_i), \\
\vec{y}^{k,i} &= [x^{k+1}_1,\ldots, x^{k+1}_i, x^k_{i+1},\ldots, x^k_n],\\
i &= 1,\ldots,n.
\end{aligned}
\end{equation}
Observe that since \(\vec{0} \in \partial J(\vec{x})\) for all \(\vec{x} \in \calC\), we still have \(\vec{p}^{k+1} \in \partial J(\vec{x}^{k+1})\). It is also straightforward to verify that the previous results and analysis of dissipative structure in this section also hold for \eqref{eq:bia2}.}

\textcolor{black}{The reason for introducing the modified scheme \eqref{eq:bia2} is so that the subgradient iterates \(p^{k}_i\) do not grow arbitrarily large when \(x^{k}_i\) equals \(u_i\) or \(l_i\) for several updates (recall that when the constraint is active, \(\partial \chi_{[l_i,u_i]}\) is unbounded). Otherwise, the iterates might get stuck at a nonstationary point for several iterations, since \(x^k_i\) is unable to leave \(\{l_i, u_i\}\) until the subgradient update in \(\partial \chi_{[l_i,u_i]}\) vanishes, yielding inefficient progress. Furthermore, this leads to pathological, albeit unlikely, examples where the iterates of \eqref{eq:bia} converge to non-stationary accumulation points. For completeness, we give such an example in \secref{sec:appendix_counter}.}

\subsection{Main convergence theorem}

\textcolor{black}{Having introduced a modified Bregman Itoh--Abe scheme in the previous section, we proceed to state and prove the main theorem of this paper, namely that all accumulation points of the scheme \eqref{eq:bia2} are nonstationary. We note that this also holds for the original Bregman Itoh--Abe method if the iterates \((\vec{x}^k)_{k\in\NN}\) converge to a unique limit.}

\begin{thm}[Stationarity guarantees]
\label{thm:main}
Let \(V\), \(J\), and \(\calC\) satisfy \assref{ass:JV}, and suppose the sequence of iterates \((\vec{x}^k,\vec{p}^k)_{k\in\NN}\) solves \eqref{eq:bia2} for time steps \((\tau_k)_{k\in\NN} \subset [\tau_{\min},\tau_{\max}]\). Then all accumulation points \(\vec{x}^* \in S\) are Clarke stationary points restricted to \(\calC\).
\end{thm}

\begin{proof}
Let \(\vec{x}^* \in S\) and consider a convergent subsequence \((\vec{x}^{k_j})_{j\in\NN}\). We want to show for each basis vector \(\vec{e}^i\) that either \(V^o(\vec{x}^*; \vec{e}^i) \geq 0\) or \(x^*_i = u_i\), and analogously that either \(V^o(\vec{x}^*; -\vec{e}^i) \geq 0 \) or \(x^*_i = l_i\). As the arguments are equivalent, we only prove the first case.

Suppose for contradiction that \(V^o(\vec{x}^*; \vec{e}^i) < -\eta\) for some \(\eta > 0\), and that \(x^*_i < u_i\). By the definition of the Clarke directional derivative, \defnref{defn:clarke_derivative}, there are \(\eps, \delta > 0\) such that for all \(\vec{x} \in B_\eps(\vec{x}^*)\) and \(\lambda \in (0,\delta)\), we have
\begin{equation}
\label{eq:eta_slope}
\frac{V(\vec{x}+\lambda \vec{e}^i) - V(\vec{x})}{\lambda} \leq -\frac{\eta}{2}.
\end{equation}
Since \(\vec{x}^{k_j} \to \vec{x}^*\) and \(\|\vec{x}^{k_j+1} - \vec{x}^{k_j}\| \to 0\), for each \(N \in \NN\) there exists \(K\) such that for all \(j \geq K\), we have \(\vec{x}^k \in B_\eps(\vec{x}^*)\) and \(\|\vec{x}^k - \vec{x}^{k+1}\| < \delta\) for \(k = k_j, k_j+1, \ldots, k_j+N\). By making \(\eps > 0\) sufficiently small, we have \(B_\eps(x^*_i) < u_i\). Furthermore, since \(x^{k+1}_i \geq x^k_i\) for \(k = k_j, \ldots, k_j + N-1\), we deduce that the constraint component \(\tilde{q}\) is zero. By \eqref{eq:eta_slope}, it follows that
\begin{align}
p^{k_j}_i - p^{k_j+N}_i &= \sum_{k = k_j}^{N-1} p^k_i - p^{k+1}_i =  \sum_{k = k_j}^{N-1} \tau^k_i \frac{V(\vec{y}^{k,i}) - V(\vec{y}^{k,i-1}))}{x^{k+1}_i - x^{k}_i} \nonumber \\
&\leq - \tau_{\min}  \sum_{k = k_j}^{N-1} \frac{\eta}{2} = - N \tau_{\min} \frac{\eta}{2}. 
\label{eq:p_inc}
\end{align}
By \assref{ass:JV}, \(\partial j_i\) is bounded on \(U = B_\eps(\vec{x}^*) \cap [l_i, u_i]\). Since \(p^{k,j}_i, \ldots, p^{k_j+N}_i \in \partial j_i(U)\), we can choose \(N\) such that \(N \tau_{\min} \frac{\eta}{2} > \max \partial j_i(U) - \min \partial j_i(U)\) and arrive at a contradiction. Thus, \(\vec{x}^*\) is a Clarke stationary point restricted to \(\calC\).
\end{proof}

\section{Examples of Bregman Itoh--Abe discrete gradient schemes}

\label{sec:examples}

In this section, we describe several schemes based on the Bregman Itoh--Abe discrete gradient scheme \eqref{eq:bia}. We will primarily consider objective functions of the form
\begin{equation}
\label{eq:linear}
V(\vec{x}) = \frac{1}{2} \inner{\vec{x},A\vec{x}} - \inner{\vec{b},\vec{x}},
\end{equation}
where \(A \in \RR^{n \times n}\) is a symmetric, positive semi-definite matrix, \textcolor{black}{with strictly positive entries \(a^i_i > 0\) on the diagonal.}

We are particularly interested in problems with underlying sparsity and/or constraints, with applications in image analysis. Throughout this section, we use a time step vector \(\vec{\tau}^k\) coordinate-wise scaled by the diagonal of \(A\), i.e. \(\vec{\tau}^k = \tau/\mbox{diag}\,(A) = [\tau/ a^1_1, \ldots, \tau /a^n_n]\) for all \(k \in \NN\), and some \(\tau > 0\).

We first introduce some well-known coordinate descent schemes for solving linear systems, which Miyatake et al. \citep{miy18} showed were equivalent to the Itoh--Abe discrete gradient method. The successive-over-relaxation (SOR) method \citep{you71} updates each coordinate sequentially according to the rule
\begin{equation}
\label{eq:SOR}
\begin{aligned}
\vec{y}^{k,0} &= \vec{x}^k \\
\vec{y}^{k,i} &= \vec{y}^{k,i-1} - \frac{\omega}{a^i_i} (\inner{\vec{a}^i, \vec{y}^{k,i-1}} - b_i) \vec{e}^i, \\
\vec{x}^{k+1} &= \vec{y}^{k,n},
\end{aligned}
\end{equation}
where \(\omega \in (0,2)\). For \(\omega = 1\), this is the Gauss-Seidel method \citep{you71}. The SOR method is equivalent to the Itoh--Abe discrete gradient method
\[
\vec{x}^{k+1} = \vec{x}^k - \vec{\tau} \overline{\nabla} V(\vec{x}^k, \vec{x}^{k+1}),
\]
with \(V\) given by \eqref{eq:linear} with the time steps \(\tau_i = 2\omega/\del{(2-\omega) a^i_i}\).

\subsection{Sparse SOR method}
\label{sec:sparseSOR}

We consider underdetermined linear systems and want to find sparse solutions \(\vec{x}^*\). Hence we seek to apply the Bregman Itoh--Abe method \eqref{eq:bia} with objective function \(V\) given by \eqref{eq:linear}, and
\begin{equation}
\label{eq:J_l1}
J(\vec{x}) = \frac{1}{2}\|\vec{x}\|^2 + \gamma \|\vec{x}\|_1,
\end{equation}
for \(\gamma  >0\). We term this the \emph{Bregman SOR (BSOR) method}.

By \propref{prop:unique}, the updates of this method are well-defined and unique. One can verify that the updates are given as follows. Denote by \(\tilde{x}^{k+1}_i\) the  standard SOR coordinate update from \(x^{k}_i\), \eqref{eq:SOR}. Furthermore, for \(\vec{p}^k \in \partial J(\vec{x}^k)\), we write \(\vec{p}^k = \vec{x}^k + \gamma \vec{r}^k\), where \(\vec{r}^k \in \partial \|\vec{x}^k\|_1\). Then \((x^{k+1}_i,r^{k+1}_i)\) are given in closed form as
\begin{equation}
\begin{aligned}
x^{k+1}_i &= S \del{\tilde{x}^{k+1}_i + \frac{2 \gamma}{2 + \tau} r^k_i,  \frac{2 \gamma}{2+\tau} },  \\
r^{k+1}_i &= r^{k}_i + \frac{\tau}{\gamma a^i_i} \del{b_i - \inner{\vec{a}^i, \vec{x}^k} - \frac{a^i_i(2+\tau)}{2\tau}  (y_i - x_i) }.
\end{aligned}
\end{equation}
Here \(S: \RR \times \RR \to \RR\) denotes the shrinkage operator \(S(\vec{x},\lambda) = \sgn(\vec{x})\max\{|\vec{x}|-\lambda, 0\}\), applied elementwise to a vector \(\vec{x}\).

\subsection{Sparse, regularised SOR}

If \(\vec{b} = A\vec{x}^{\text{true}} + \vec{\delta}\), where \(\vec{x}^{\text{true}}\) is the sparse ground truth and \(\vec{\delta}\) is noise, then it may be necessary to regularise the objective function as well. Hence we consider the objective function
\begin{equation}
\label{eq:V_L1}
V(\vec{x}) = \frac{1}{2} \inner{\vec{x},A\vec{x}} - \inner{\vec{b},\vec{x}} + \lambda \|\vec{x}\|_1,
\end{equation}
for some regularisation parameter \(\lambda > 0\). The non-smoothness induced by \(\|\cdot\|_1\) satisfies \assref{ass:JV}, so \thmref{thm:main} implies Bregman Itoh--Abe discrete gradient methods will still converge to stationary points of this problem.

For both \(J(\vec{x}) = \frac{1}{2}\|\vec{x}\|^2 + \gamma \|\vec{x}\|_1\) and \(J(\vec{x}) = \frac{1}{2} \|\vec{x}\|^2\), the scheme \eqref{eq:bia} can be expressed in closed	 form for \eqref{eq:V_L1}, on a case-by-case basis. However, for purposes of brevity, we leave include this in the appendix.

\section{Equivalence of iterative methods for linear systems}

In what follows, we discuss and demonstrate equivalencies for different iterative methods for solving linear systems. We recall from the previous section that  the SOR method \eqref{eq:SOR} is equivalent to the Itoh--Abe discrete gradient method \citep{miy18}.

The explicit coordinate descent method \citep{bec13, wri15} for minimising \(V\) is given by
\begin{equation}
\label{eq:coordinate_descent}
\begin{aligned}
\vec{y}^{k,0} &= \vec{x}^k \\
\vec{y}^{k,i} &= \vec{y}^{k,i-1} - \alpha_i [\nabla V(\vec{y}^{k,i-1})]_i \vec{e}^i, \\
\vec{x}^{k+1} &= \vec{y}^{k,n},
\end{aligned}
\end{equation}
where \(\alpha_i > 0\) is the time step. As mentioned in \citep{wri15}, the SOR method is also equivalent to the coordinate descent method with \(V\) given by \eqref{eq:linear} and the time step \(\alpha_i = \omega/a^i_i\). Hence, in this setting, the Itoh--Abe discrete gradient method is equivalent not only to SOR methods, but to explicit coordinate descent.

It is not surprising that these iterative coordinate methods turn out to be the same, given that the gradient \(V\) in \eqref{eq:linear} is linear. Furthermore, these equivalencies extend to discretisations of the inverse scale space flow with \(J\) given by \eqref{eq:J_l1}. The resultant Bregman Itoh--Abe scheme for \eqref{eq:linear} is described in \secref{sec:sparseSOR}. We may compare this to a \emph{Bregman linearised coordinate descent} scheme,
\begin{equation*}
\begin{aligned}
\vec{y}^{k,0} &= \vec{x}^k, \quad \vec{p}^k \in \partial J(\vec{x}^k) \\
z_i &= \argmin_{y} [\nabla V(\vec{y}^{k,i-1})]_i \cdot y + \frac{a^i_i}{\alpha_i} D^{\vec{p}^k}_J(\vec{y}^{k,i-1},\vec{y}^{k,i-1}+y\vec{e}^i), \\
\vec{y}^{k,i} &= \vec{y}^{k,i-1} + z_i \vec{e}^i, \\
\vec{x}^{k+1} &= \vec{y}^{k,n}.
\end{aligned}
\end{equation*}
One can verify that this scheme is equivalent to \eqref{eq:coordinate_descent} for the parameters
\[
\tau_i = \frac{2\alpha}{(2-\alpha) a^i_i}, \quad \lambda^* = \frac{\lambda}{1 + \frac{\alpha}{2-\alpha}}.
\]

\section{Numerical examples}
\label{sec:numerics}

In this section, we present numerical results for the schemes described in \secref{sec:examples}.

\subsection{Sparse SOR}

We construct a matrix \(A \in \RR^{1024 \times 1024}\) from independent standard (zero mean, unit variance) Gaussian draws, and construct the sparse ground truth \(\vec{x}^{\text{true}}\) by choosing \(10\%\) of the indices at random determined by uniform draws on the unit interval. We then solve the problem \[
\argmin_{\vec{x}} \frac{1}{2}\|A\vec{x} - \vec{b}\|^2,
\]
where \(\vec{b} = A\vec{x}^{\text{true}}\). We compare the SOR method (\(J(\vec{x}) = \|\vec{x}\|^2/2\)) and the BSOR method (\(J(\vec{x}) = \|\vec{x}\|^2/2 + \gamma \|\vec{x}\|_1\)), where \(\gamma = 1\). We set time steps to \(\vec{\tau} = 2/\mbox{diag}\,(A)\), corresponding to the Gauss-Seidel method. See \figref{fig:gaussian_noiseless} for the results.

\begin{figure}
\begin{subfigure}{0.98\columnwidth}
\begin{center}
\includegraphics[width=0.8\columnwidth]{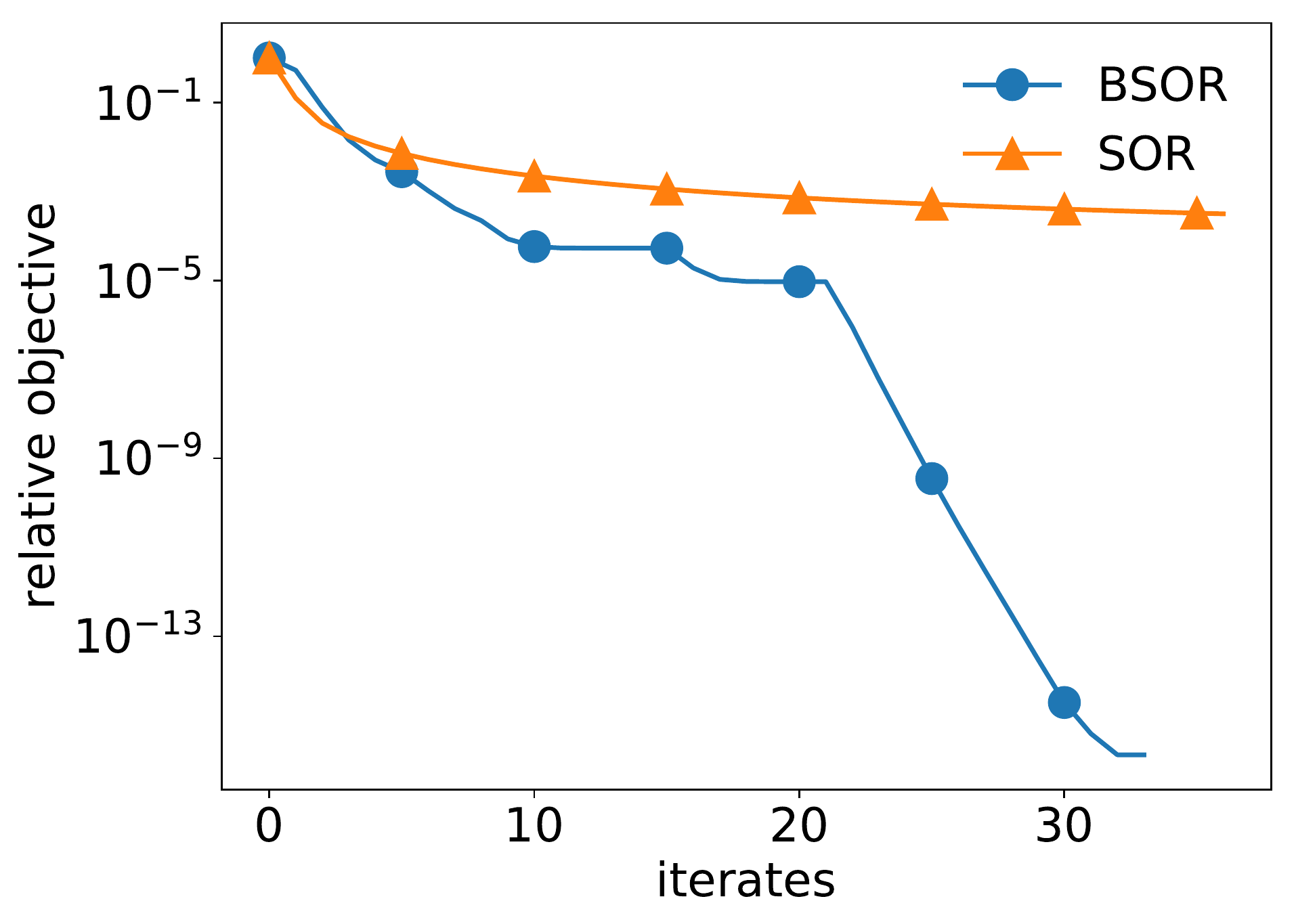}
\end{center}
\end{subfigure}
\begin{subfigure}{0.98\columnwidth}
\begin{center}
\includegraphics[width=0.8\columnwidth]{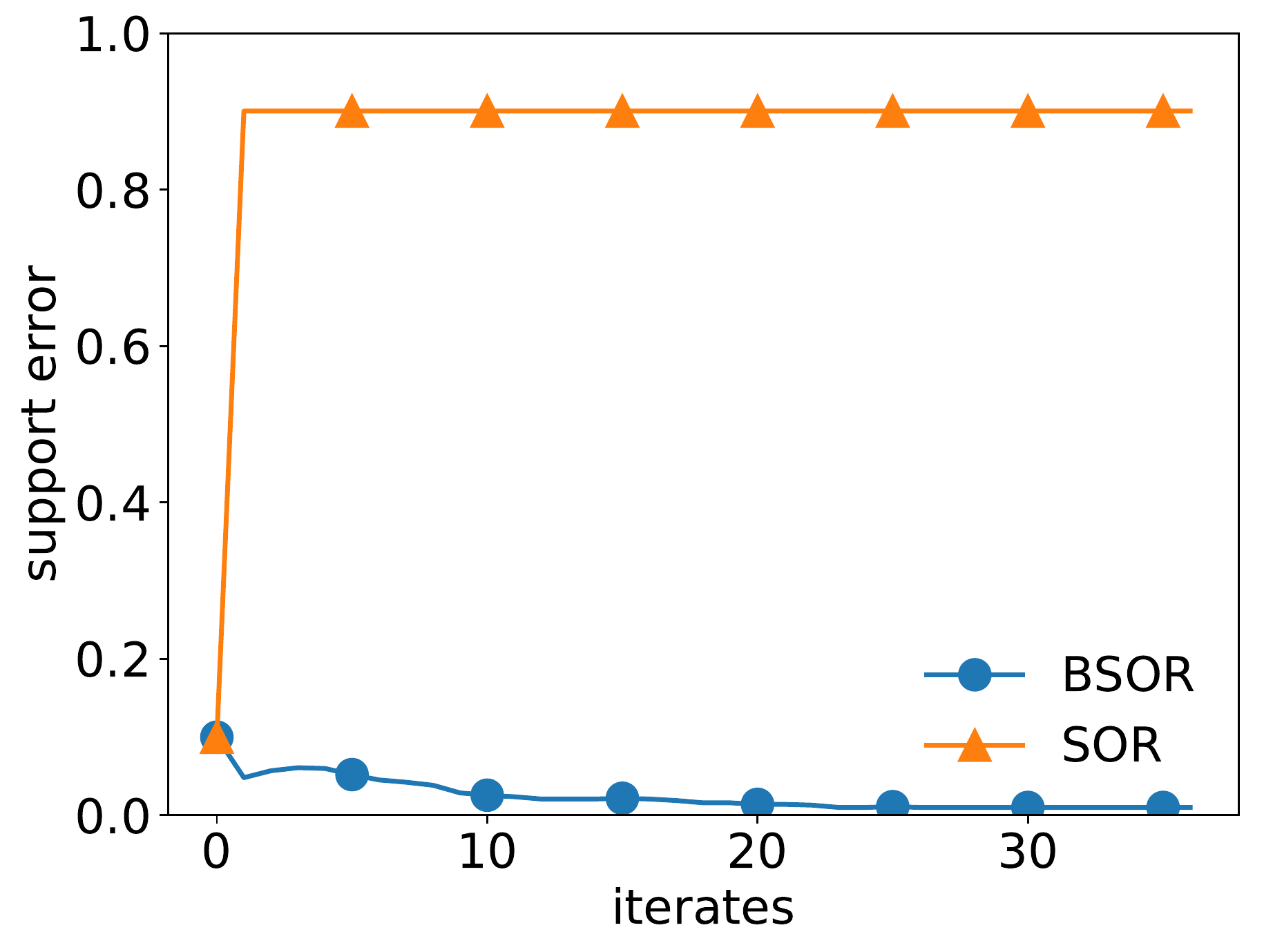}
\end{center}
\end{subfigure}
\caption{Comparison of SOR and sparse SOR methods, for Gaussian linear system without noise. Top: Convergence rate for relative objective, i.e. \([V(\vec{x}^k) - V^*]/[V(\vec{x}^0)-V^*]\). Bottom: Support error with respect to iterates, i.e. proportion of indices \(i\) s.t. \(\sgn(x^k_i)=\sgn(x^*_i)\).}
\label{fig:gaussian_noiseless}
\end{figure}

For the same test problem, but where the ground truth is binary, i.e. only takes values \(1\) or 0, see \figref{fig:gaussian_noiseless_alt}.
\begin{figure}
\begin{subfigure}{0.98\columnwidth}
\begin{center}
\includegraphics[width=0.8\columnwidth]{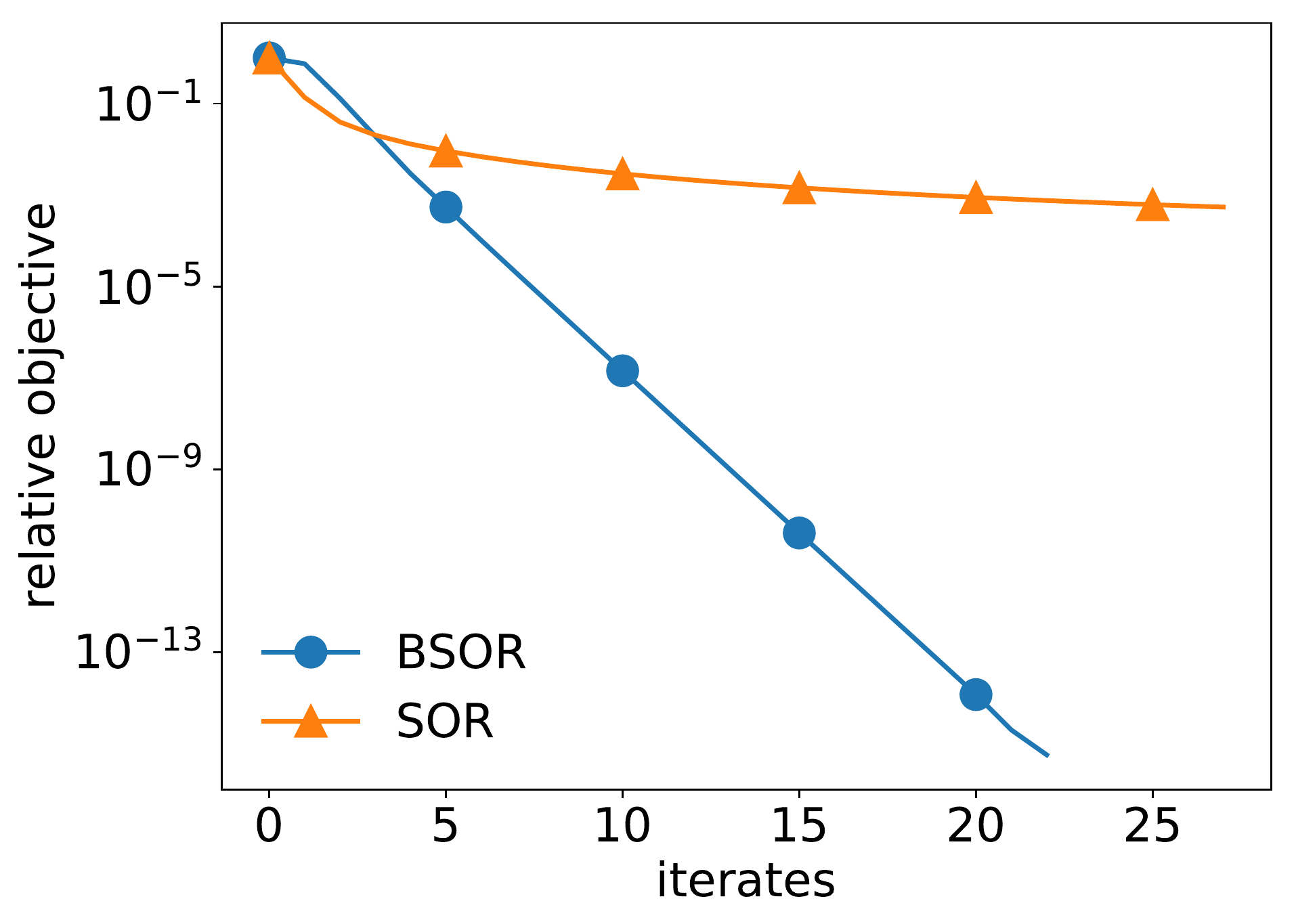}
\end{center}
\end{subfigure}
\begin{subfigure}{0.98\columnwidth}
\begin{center}
\includegraphics[width=0.8\columnwidth]{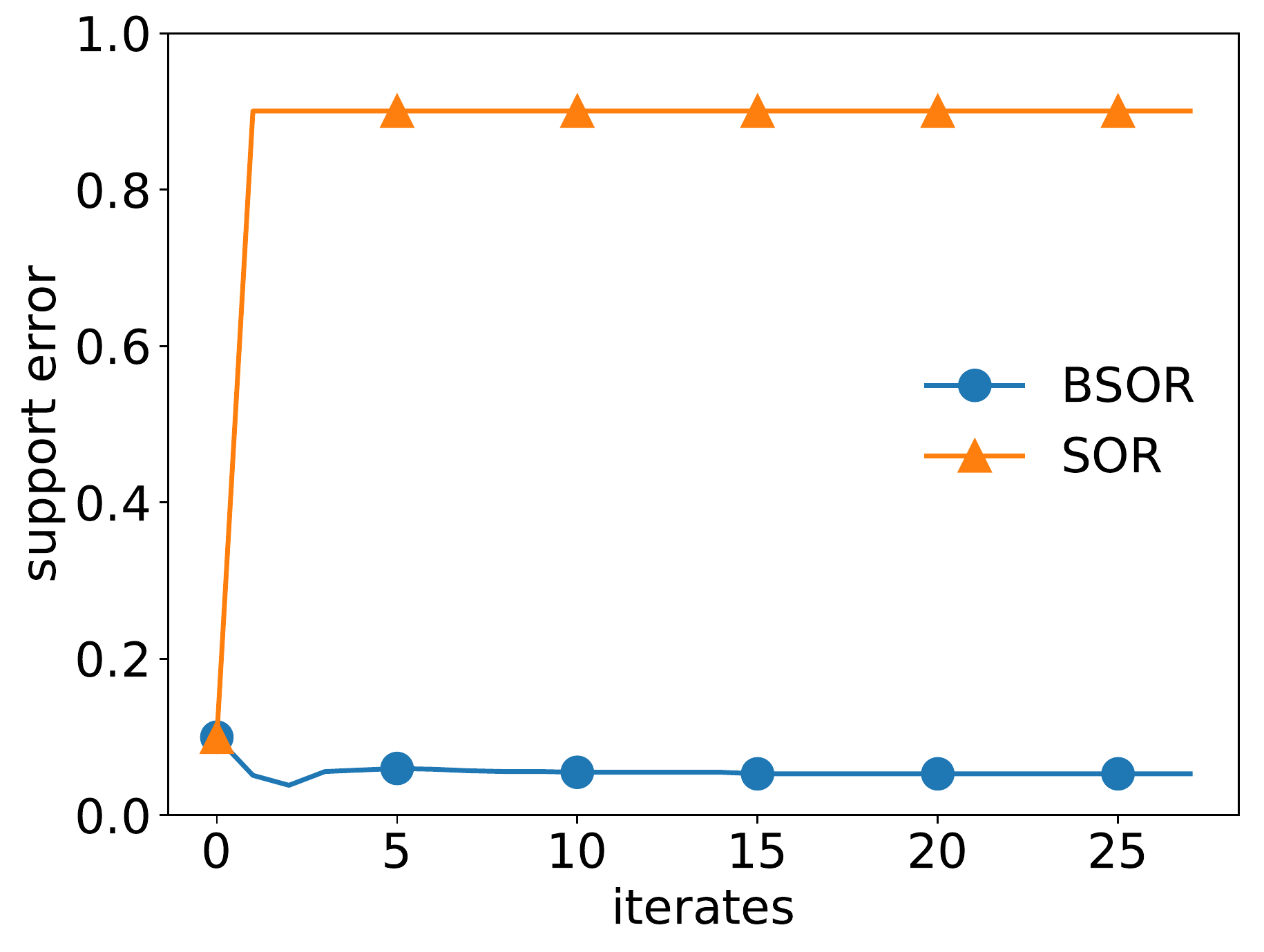}
\end{center}
\end{subfigure}
\caption{Comparison of SOR and sparse SOR methods, for Gaussian linear system without noise, and binary ground truth. Top: Convergence rate for relative objective. Bottom: Support error with respect to iterates.}
\label{fig:gaussian_noiseless_alt}
\end{figure}

\subsection{Sparse, regularised SOR}

We construct \(A \in \RR^{1024 \times 1024}\) and \(\vec{x}^{\text{true}}\) as in the previous subsection. However, we add noise to the data, i.e. \(\vec{b} = A\vec{x}^{\text{true}} + \vec{\delta}\), where \(\vec{\delta}\) is independent Gaussian noise with a standard deviation of \(0.1 \|A\vec{x}^{\text{true}}\|_\infty\). Since the added noise destroys the sparsity structure of \(A^{-1}\vec{b}\), the sparse SOR method fails to improve the convergence rate. The results for \(V(\vec{x}) = \|A\vec{x}-\vec{b}\|^2/2\) are in \figref{fig:gaussian_noisy}. 

We therefore include regularisation in the objective function of the form
\[
V(\vec{x}) = \frac{1}{2} \|A\vec{x}-\vec{b}\|^2/2 + \lambda \|\vec{x}\|_1,
\]
where \(\lambda = 100\), and with initialisation \(\vec{x}^0\) constructed by random, independent Gaussian draws. The results are visualised in \figref{fig:gaussian_noisy_l1}.

\begin{figure}
\begin{subfigure}{0.98\columnwidth}
\begin{center}
\includegraphics[width=0.8\columnwidth]{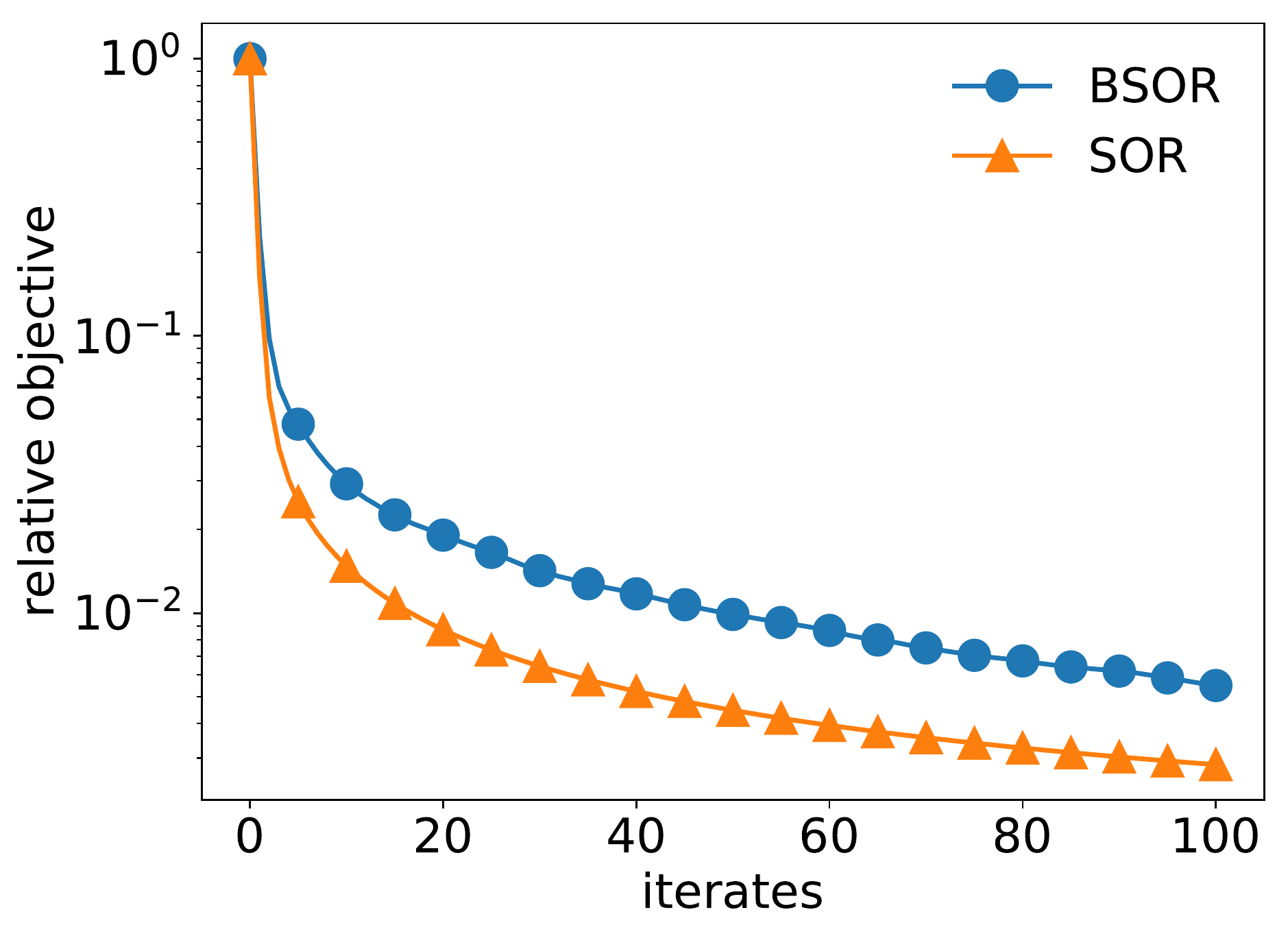}
\end{center}
\end{subfigure}
\begin{subfigure}{0.98\columnwidth}
\begin{center}
\includegraphics[width=0.8\columnwidth]{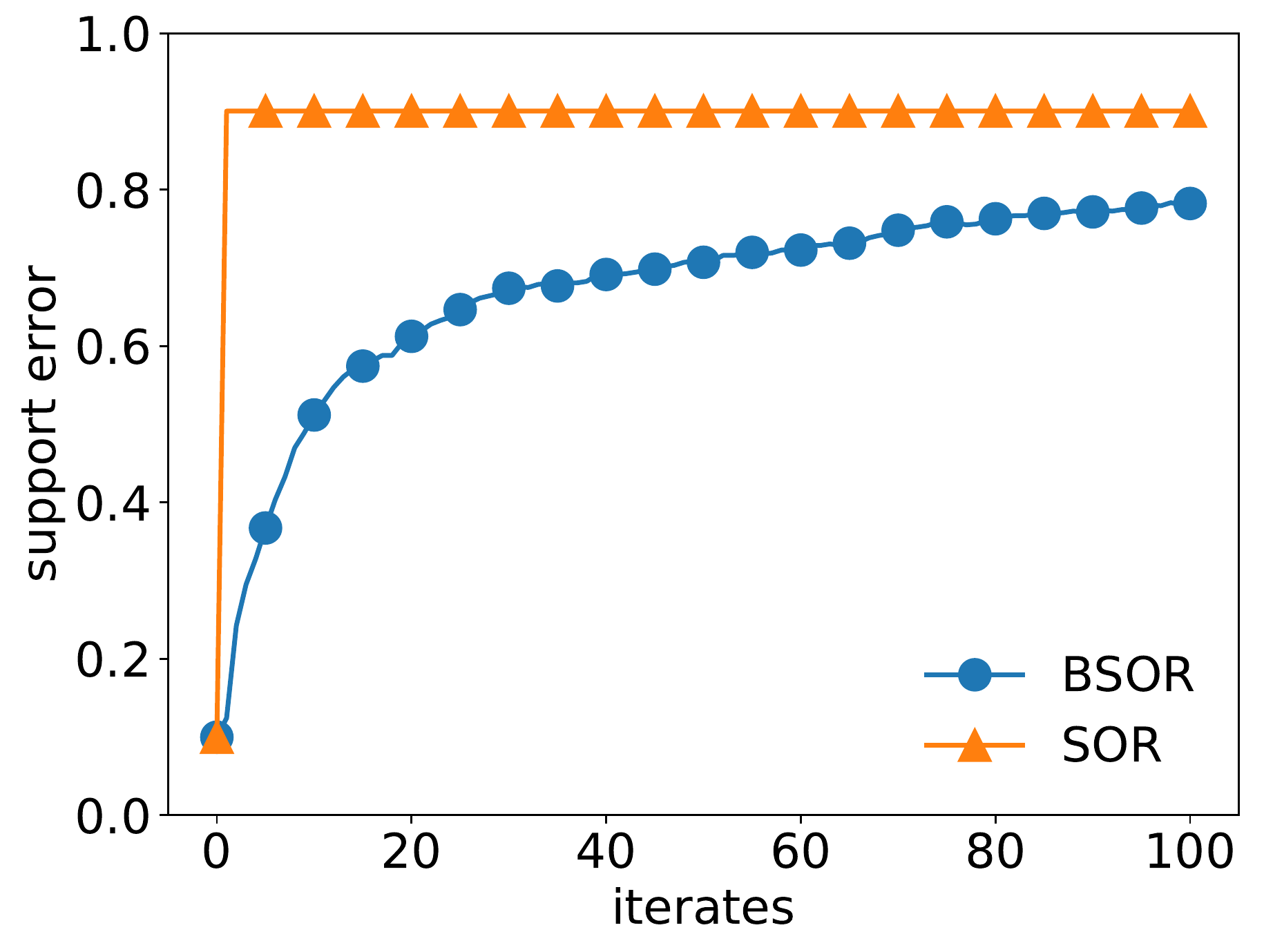}
\end{center}
\end{subfigure}
\caption{Comparison of SOR and sparse SOR methods, for Gaussian linear system with noise. Top: Convergence rate for relative objective. Bottom: Support error with respect to iterates.}
\label{fig:gaussian_noisy}
\end{figure}

\begin{figure}
\begin{subfigure}{0.98\columnwidth}
\begin{center}
\includegraphics[width=0.8\columnwidth]{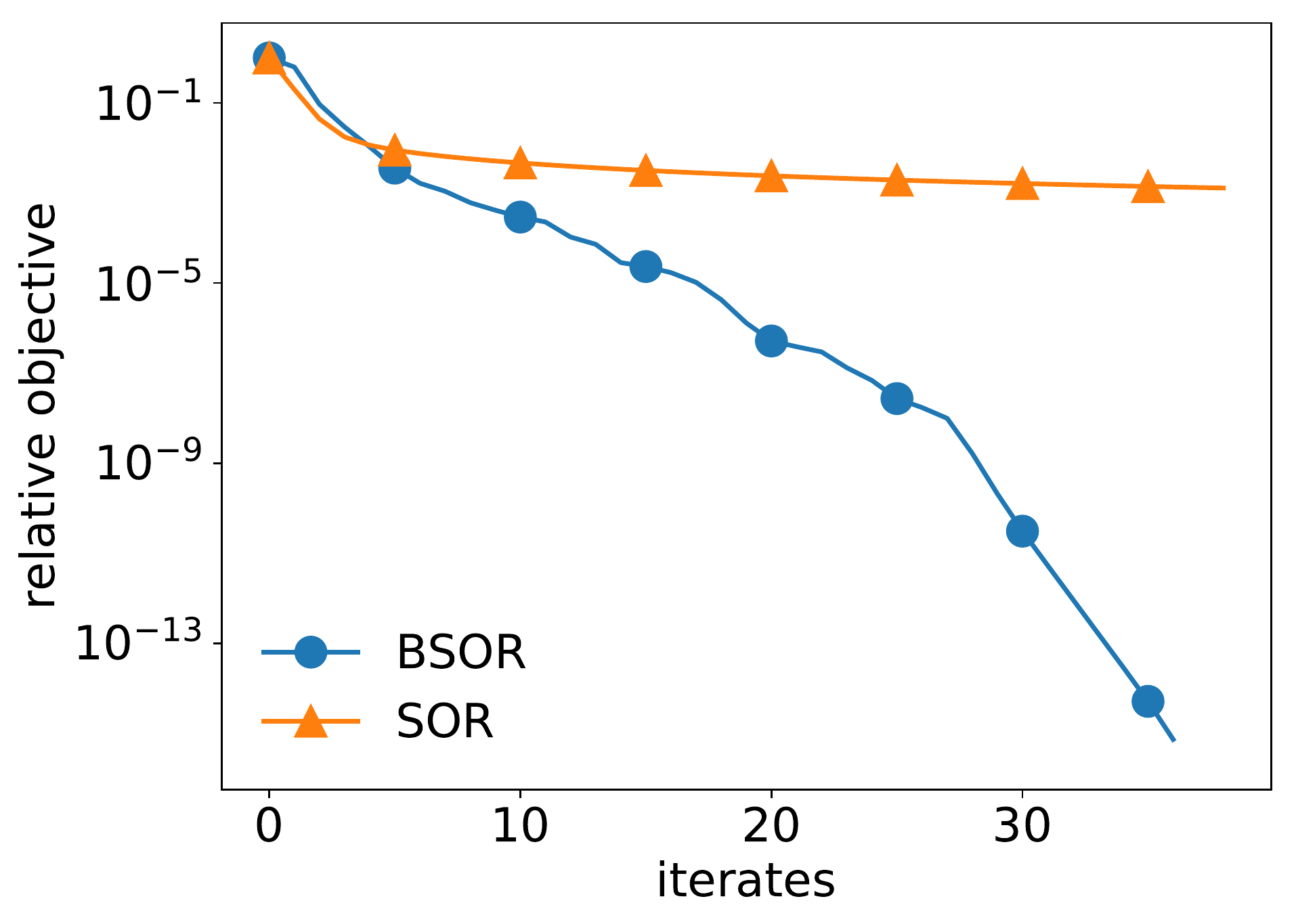}
\end{center}
\end{subfigure}
\begin{subfigure}{0.98\columnwidth}
\begin{center}
\includegraphics[width=0.8\columnwidth]{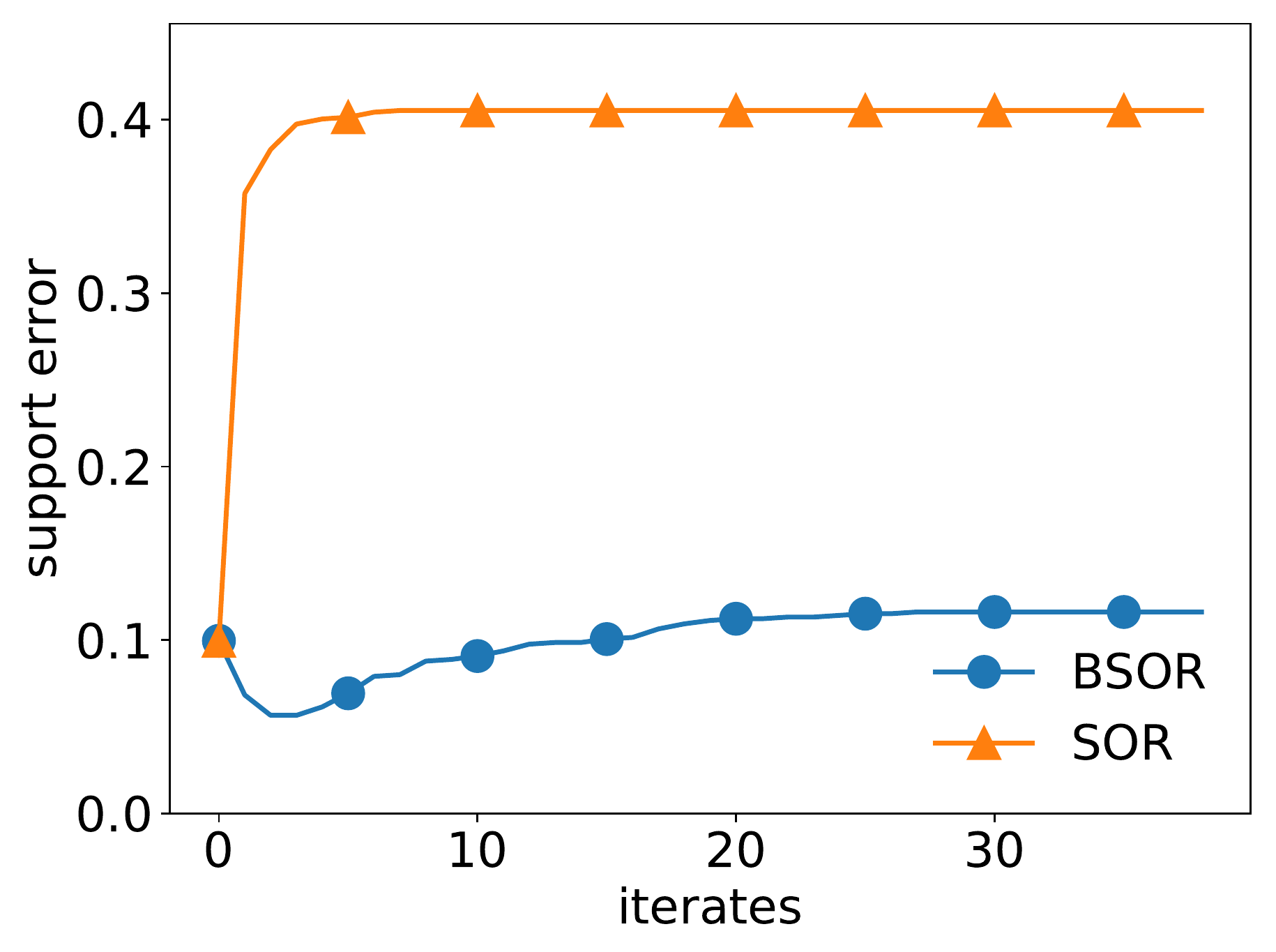}
\end{center}
\end{subfigure}
\caption{Comparison of SOR and sparse SOR methods, for  \(\ell^1\)-regularised linear system with noise. Top: Convergence rate for relative objective. Bottom: Support error with respect to iterates.}
\label{fig:gaussian_noisy_l1}
\end{figure}

\subsection{Student-t regularised image denoising}

\textcolor{black}{We consider a nonconvex image denoising model, previously presented in \citep{och14}, given by
\begin{equation}
\label{eq:nonconvex_problem}
F: \RR^n \to \RR, \quad F(\vec{x}) := \sum_{i=1}^N \varphi_i \Phi(K_i \vec{x}) + \|\vec{x} - \vec{x}^\delta\|_1.
\end{equation}
Here \(\{K_i\}_{i=1}^N\) is a collection of linear filters, \((\varphi_i)_{i=1}^N \subset [0,\infty)\) are coefficients, \(\Phi: \RR^n \to \RR\) is the nonconvex function based on the t-student distribution, defined as
\[
\Phi(x) := \sum_{j=1}^n \psi(x_i), \quad \psi(x) := \log(1+x^2),
\]
and \(x^\delta\) is an image corrupted by impulse noise (salt \& pepper noise).}

\textcolor{black}{As impulse noise only affects a small subset of pixels, we use the data fidelity term \(\vec{x} \mapsto \|\vec{x} - \vec{x}^\delta\|_1\) to promote sparsity of \(\vec{x}^*-\vec{x}^\delta\) for \(\vec{x}^* \in \argmin F(\vec{x})\). As linear filters, we consider the simple case of finite difference approximations to first-order derivatives of \(\vec{x}\). We note that by applying a gradient flow to this regularisation function, we observe a similarity to Perona-Malik diffusion \citep{per90}.}

\textcolor{black}{We consider a Bregman Itoh--Abe method (abbreviated to BIA) for solving \(\min_{\vec{x} \in \RR^n} F(\vec{x})\) with the Bregman function
\[
J(\vec{x}) := \frac{1}{2}\|\vec{x}\|^2 + \gamma \|\vec{x} - \vec{x}^\delta\|_1,
\]
to account for the sparsity of the residual \(\vec{x}^* - \vec{x}^\delta\), and compare the method to the regular Itoh--Abe discrete gradient method (abbreviated to IA).}

\textcolor{black}{We set the starting point \(\vec{x}^0 = \vec{x}^\delta\), and the parameters to \(\tau_k=1\) for all \(k\), \(\gamma = 0.5\), and \(\varphi_i = 2\), \(i =1,2\).  For the impulse noise, we use a noise density of \(10 \%\). In the case where \(x^{k+1}_i\) is not set to \(x^\delta_i\), we use the scalar root solver \textit{scipy.optimize.brenth} on Python. Otherwise, the updates are in closed form.}

\textcolor{black}{See \figref{fig:mrf} for numerical results. By gradient norm, we mean \(\dist(\partial F(\vec{x}^k),0)\). We observe that as in the convex case, the Bregman Itoh--Abe method converges significantly faster, due to the sparse structure of the problem.}

\begin{figure}
\begin{subfigure}{0.98\columnwidth}
\begin{center}
\includegraphics[width=0.8\columnwidth]{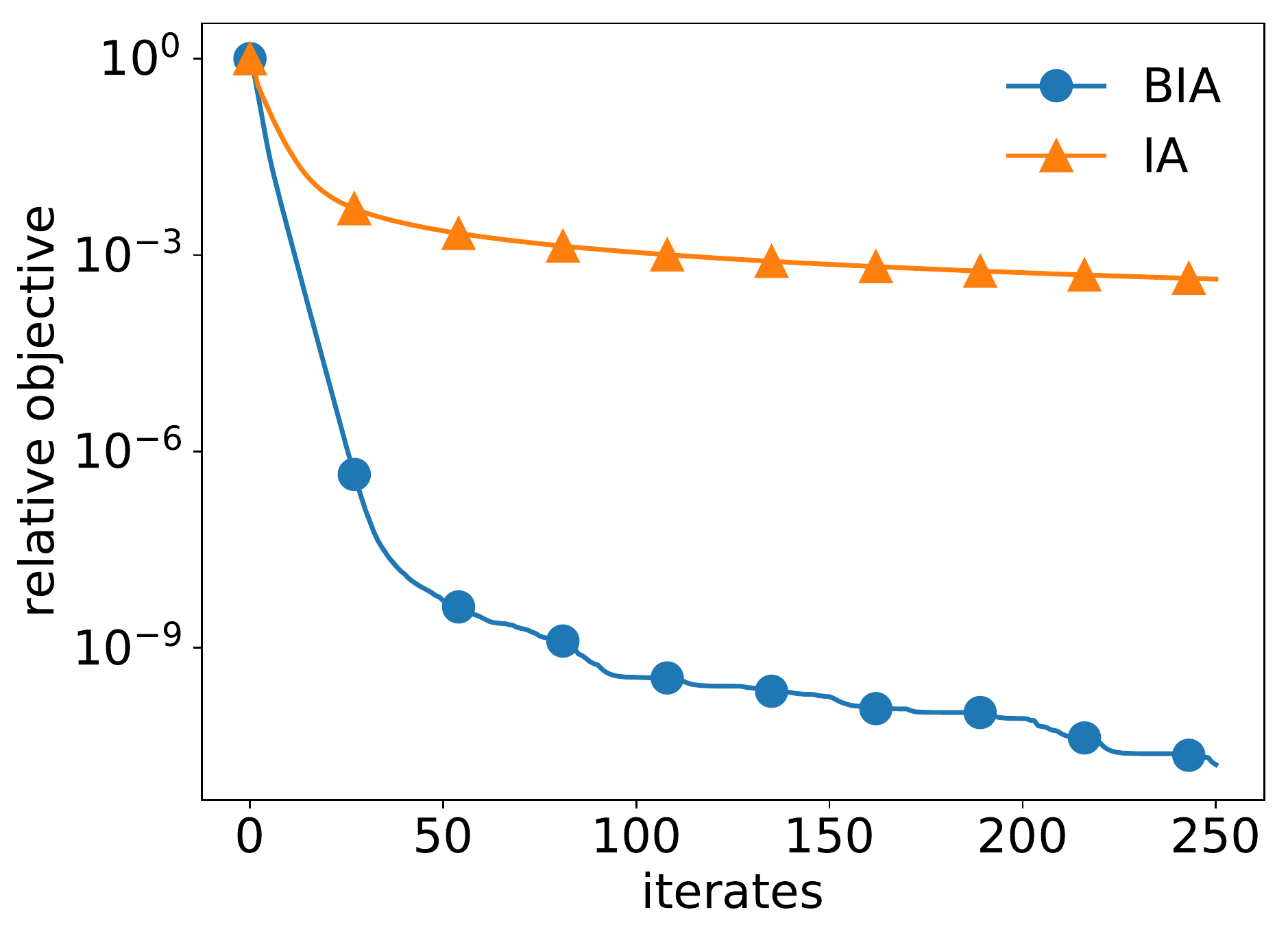}
\end{center}
\end{subfigure}
\begin{subfigure}{0.98\columnwidth}
\begin{center}
\includegraphics[width=0.8\columnwidth]{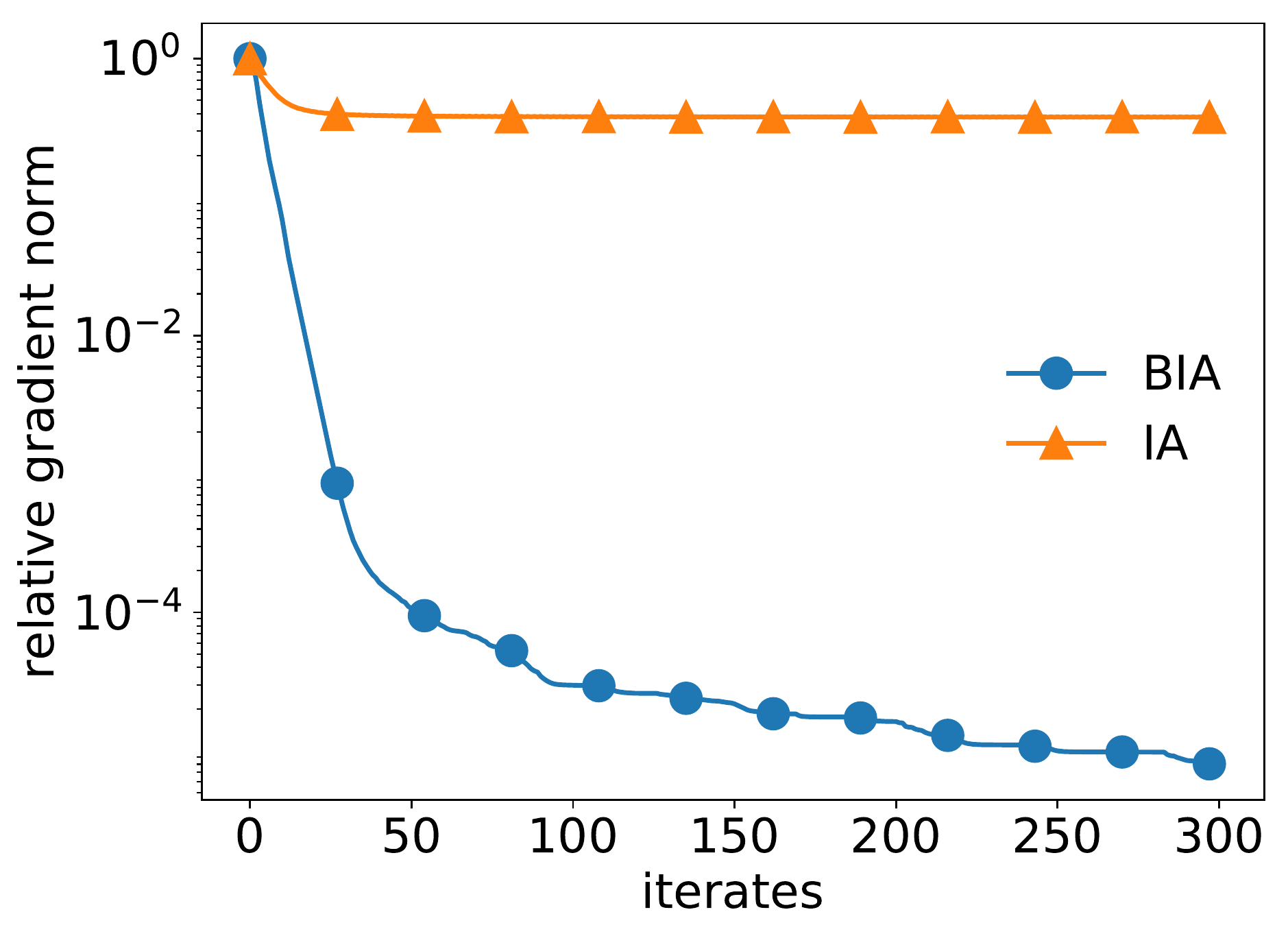}
\end{center}
\end{subfigure}
\begin{subfigure}{0.98\columnwidth}
\begin{center}
\includegraphics[width=0.8\columnwidth]{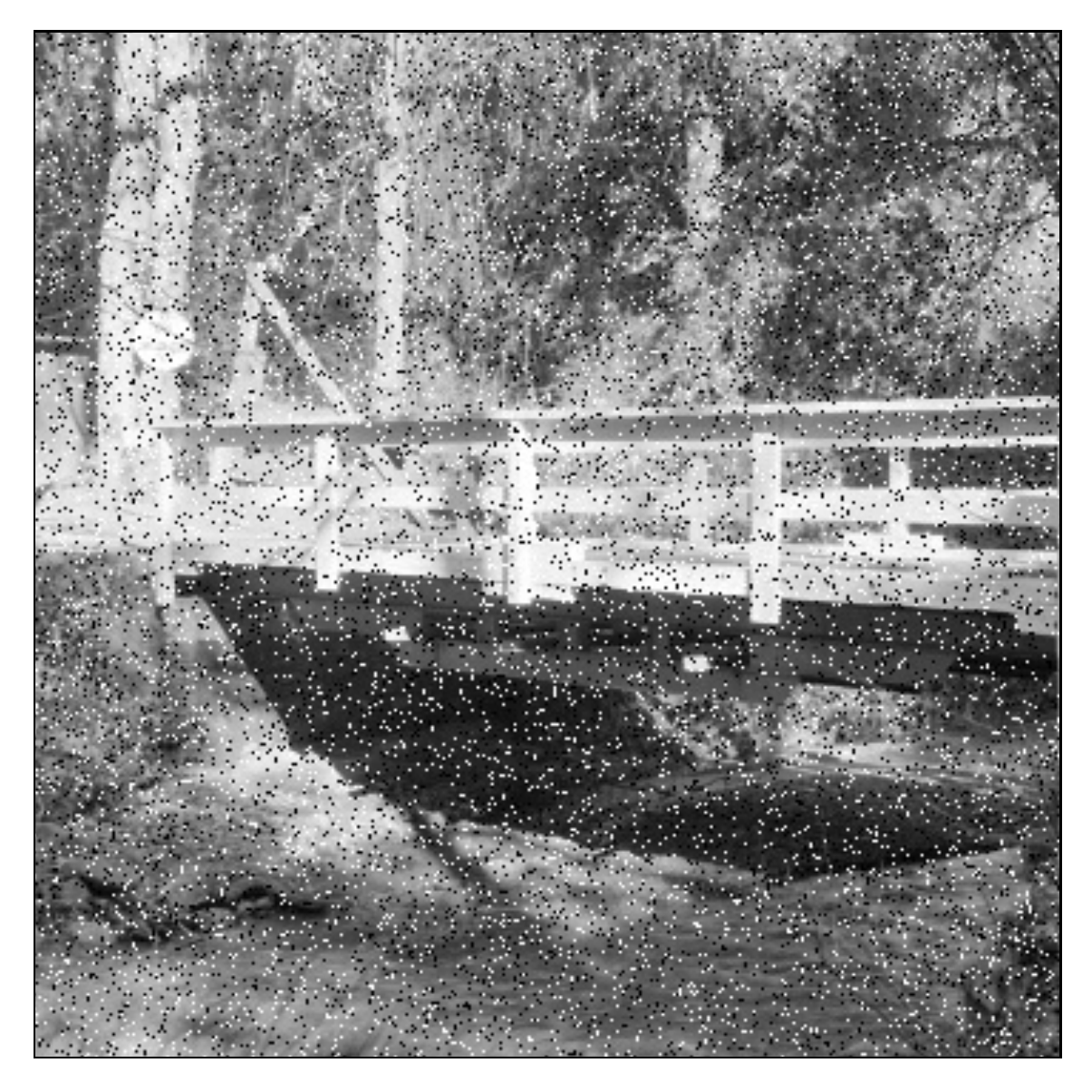}
\end{center}
\end{subfigure}
\begin{subfigure}{0.98\columnwidth}
\begin{center}
\includegraphics[width=0.8\columnwidth]{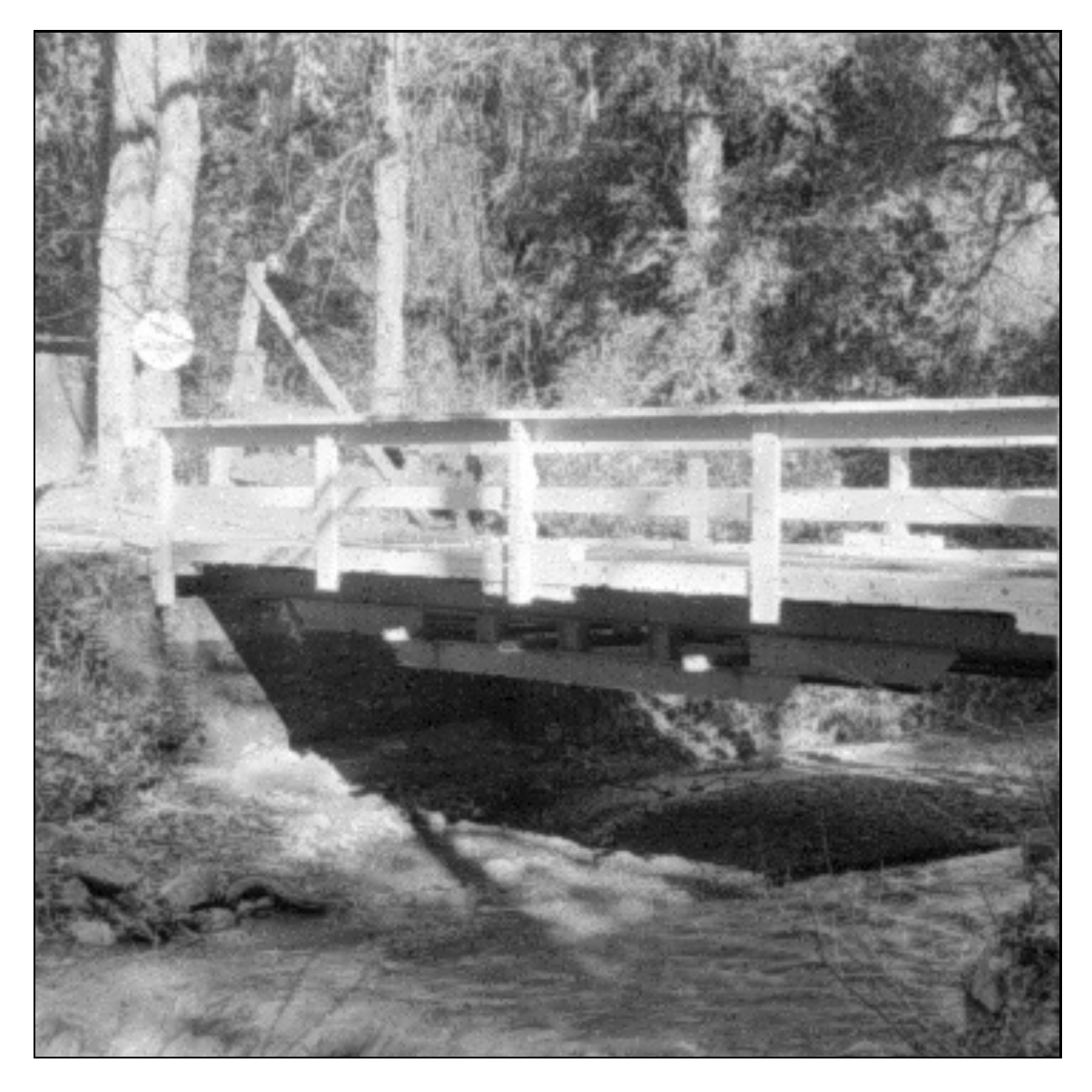}
\end{center}
\end{subfigure}
\caption{\textcolor{black}{Comparison of BIA and IA methods, for student-t regularised image denoising. First: Convergence rate for relative objective. Second: Convergence rate for relative gradient norm. Third: Input data. Fourth: Reconstruction.}}
\label{fig:mrf}
\end{figure}

\textcolor{black}{In all of these cases, the sparsity structure when utilised properly leads to significantly faster convergence rates with the BSOR method.}

\section{Conclusion}

In this paper, we propose to discretise the ISS flow with the Itoh--Abe discrete gradient. The resultant schemes exhibit a dissipative structure \eqref{eq:bdg_dissipation} related to the symmetrised Bregman distance of a function \(J\). This generalises the discrete gradient method for gradient flows, and can be viewed as a discrete gradient analogue to Bregman iterations. Building on previous studies of the Itoh--Abe discrete gradient method in the non-smooth, non-convex setting, we prove convergence guarantees of the Bregman Itoh--Abe discrete gradient method in such a setting.

We consider numerical examples motivated by linear systems and searching for sparse solutions. These results indicate that for sparse reconstructions, popular iterative solvers such as the SOR method can be significantly sped up by incorporating a Bregman step.

Future work is dedicated to proving convergence rates for the Bregman Itoh--Abe methods, and to compare the scheme to related methods such as the sparse Kaczmarz method \citep{lor14sfp}. Furthermore, we will study corresponding inverse scale space schemes using other discrete gradients, such as the mean value discrete gradient.

\begin{appendices}

\section{Counterexample of \thmref{thm:main} for \eqref{eq:bia}}
\label{sec:appendix_counter}

\textcolor{black}{In the following example, we describe an example of a constrained optimisation problem for which the iterates of the unmodified Bregman Itoh--Abe method \eqref{eq:bia} fails to converge to a limit set of stationary points.}
\begin{ex}
\textcolor{black}{As a starting point, we consider a \(C^2\)-smooth objective function \(W : \RR^2 \to \RR\) as described by Curry \citep{cur44}, for which the trajectory of a Euclidean gradient flow spirals along a gully, asymptotically trending towards the unit circle \(S^1 = \{[x,y]^T \in \RR^2 \; : \; x^2 + y^2 = 1\}\).}

\textcolor{black}{Denote by \([\tilde{x}^k, \tilde{y}^k]^T\) the iterates from the (standard) Itoh--Abe method for \(W\) with (arbitrary) time steps \(\tau \equiv 1\) and starting point \([\tilde{x}^0, \tilde{y}^0]^T\) in the gully. As each update of the method consists of moving to a local point of descent, the iterates of this method will also remain in this gully and spiral towards \(S^1\).}

\textcolor{black}{We assume that there exists compact, disjoint sets \(A, B \subset \RR^2\) with nonempty interior, such that \(B \cap S^1 \neq \emptyset\), and such that most of the iterates \(([\tilde{x}^k, \tilde{y}^k]^T)_{k\in\NN}\) are in \(A\). That is, for all \(k \in \NN\), the set \(\{i \leq k \; : \; [\tilde{x}^i,\tilde{y}^i]^T \in A\}\) is strictly greater than the set \(\{i \leq k \; : \; [\tilde{x}^i,\tilde{y}^i]^T \notin A\}\). Such a set \(A\) exists as long as the iterates are spiralling around the unit circle at approximately a steadily decreasing rate.}

\textcolor{black}{Next we define an objective function \(V: \RR^3 \to \RR\) as folllows. Let \(f_A, f_B: \RR^2 \to \RR\) be smooth support functions of \(A\) and \(B\) respectively, with disjoint support sets. That is, \(f_A(x,y) \in [0,1]\),  \(f_B(x,y) \in [0,1]\), and \(f_A(x,y) \cdot f_B(x,y) = 0\) for all \([x,y]^T \in \RR^2\), and furthermore, \(f_A(A) \equiv 1\) and \(f_B(B) \equiv 1\). We then consider the optimisation problem
\begin{equation}
\label{eq:counter}
\min_{z \geq 0 }V(x,y,z) := W(x,y) + z \del{f_A(x,y) - f_B(x,y)}.
\end{equation}
We implement the unmofidied Bregman Itoh--Abe method \eqref{eq:bia} with \(\tau \equiv 1\), \(\vec{x}^0 = [\tilde{x}^0, \tilde{y}^0, 0]^T\), and \(J(\vec{x}) = \|\vec{x}\|^2/2 + \chi_{[z\geq 0]}(z)\), where we use the shorthand notation \(\vec{x} = [x,y,z]^T\).}

\textcolor{black}{If \([x,y]^T \in B \cap S^1\), then \(\partial_z V(x,y,0) = -1\), meaning that \([x,y,0]^T\) is a nonstationary point of \(V\) restricted to \(\{z \geq 0\}\). Since the limit set of \(([\tilde{x}^k, \tilde{y}^k]])_{k\in\NN}\) is \(S^1\), it follows that the iterates \(([\tilde{x}^k, \tilde{y}^k,0])_{k\in\NN}\) admit nonstationary accumulation points for \(V\). It therefore remains to show that for all \(k \in \NN\), \(\vec{x}^k = [\tilde{x}^k, \tilde{y}^k,0]\) are admissible updates to the Bregman Itoh--Abe method \eqref{eq:bia} for \eqref{eq:counter}.}

\textcolor{black}{To verify that \(\vec{x}^k = [\tilde{x}^k, \tilde{y}^k,0]\) for all \(k\), we can argue by induction. If \(\vec{x}^k = [\tilde{x}^k, \tilde{y^k},0]\), then \(V(\tilde{x}^k, \tilde{y}^k, 0) = W(\tilde{x}^k, \tilde{y}^k)\) so \(\tilde{x}^{k+1}\) and \(\tilde{y}^{k+1}\) are admissible updates for the \(x\)- and \(y\)-coordinates.}

\textcolor{black}{It remains to verify the same for the \(z\)-coordinate as well. If \([\tilde{x}^{k+1}, \tilde{y}^{k+1}]^T \in A\), then  \(\partial_z V(\tilde{x}^{k+1}, \tilde{y}^{k+1}, 0) = 1\), so we have the update \(z^{k+1} = 0\) and \(p^{k+1}_3 = p^k_3 - 1 \in \chi_{[z \geq 0]}(z^{k+1})\). On the other hand, if \([\tilde{x}^{k+1}, \tilde{y}^{k+1}]^T \notin A\), then  \(\partial_z V(\tilde{x}^{k+1}, \tilde{y}^{k+1}, 0) \in [-1,1]\), so if in addition \(p^k_3 < -1\), then we have the update \(z^{k+1} = 0\), \(p^{k+1}_3 = p^k_3 - \partial_z V(\tilde{x}^{k+1}, \tilde{y}^{k+1}, 0)\). Denoting by \(M\) and \(N\) the cardinalities of  \(\{i \leq k \; : \; [\tilde{x}^i,\tilde{y}^i]^T \in A\}\) and \(\{i \leq k \; : \; [\tilde{x}^i,\tilde{y}^i]^T \notin A\}\) respectively, we know by construction of \(A\) that \(M - N \geq 1\). Furthermore, we observe that \(p^k_3 \leq N-M \leq -1\), which concludes the argument.}
\end{ex}

\section{$\ell^1$-regularised sparse SOR method}

In what follows, we describe the update rule for the Bregman Itoh--Abe method with \(V\) given in \eqref{eq:V_L1} and \(J\) given in \eqref{eq:J_l1}.

Denote by \(\tilde{x}^{k+1}_i\) the standard SOR update \eqref{eq:SOR} for the \(i\)th coordinate. Then the \(\ell^1\)-regularised sparse SOR method can be expressed as follows.
\begin{enumerate}
	\item If \(x^k_i = 0\) and \(|\tilde{x}^{k+1}_i - \gamma r^{k}_i| \leq \gamma + \lambda \tau/a^i_i\), then
	\[
	x^{k+1}_i = 0, \quad r^{k+1}_i = \frac{\gamma r^k_i - \tilde{x}^{k+1}_i}{\gamma+\lambda \tau/a^i_i}.
	\]
	\item Else if 
	\[
	|(\tau/2 +1) \tilde{x}^{k+1}_i + \gamma r^k_i  | \geq \gamma + \tau \lambda/a^i_i,
	\]
	then 
	\begin{align*}
	x^{k+1}_i &= \tilde{x}^{k+1}_i + \frac{\gamma r^k_i -\del{\gamma + \tau \lambda/a^i_i} \sgn\del{ \tilde{x}^{k+1}_i + \frac{\gamma}{\tau/2 + 1} r^k_i}}{\tau/2 + 1} \\
	r^{k+1}_i &= \sgn\del{ \tilde{x}^{k+1}_i + \frac{\gamma r^k_i}{\del{\tau/2 +1} }  }.
	\end{align*}
	\item Else if \(x^k_i \neq 0\) and
	\[
	\envert{\del{\tau/2 + 1} \tilde{x}^{k+1}_i + \gamma r^k_i  - (\lambda \tau/a^i_i) \sgn(x^k_i)} \leq \gamma,
	\]
	then set
	\begin{align*}
	x^{k+1}_i &= 0, \\
	r^{k+1}_i &= r^k_i + \frac{1}{\gamma} \del{\del{\tau/2 + 1} \tilde{x}^{k+1}_i - (\tau \lambda/a^i_i) \sgn(x^k_i) }.
	\end{align*}
	\item Else if \(x_i \neq 0\) and
	\begin{align*}
	&\envert{2 \del{\frac{a^i_i }{2} + \frac{a^i_i }{\tau}} \tilde{x}^{k+1}_i + \del{\frac{2 a^i_i\gamma}{\tau} +\lambda} \sgn(x^k_i) }^2 \\
	&\leq\del{b_i - \inner{\vec{a}^i,\vec{y}^{k,i-1}} + \del{\frac{2a^i_i \gamma}{\tau} +\lambda} \sgn(x^k_i)}^2\ldots \\
	 &\qquad\qquad+8\lambda \del{\frac{a^i_i}{2} + \frac{a^i_i}{\tau}}|x^k_i|,
	\end{align*}
	then set
	{\small
	\begin{align*}{\scriptstyle
	&x^{k+1}_i = \tilde{x}^{k+1}_i + \frac{\sgn(x^k_i)}{2\del{\frac{a^i_i}{2} + \frac{a^i_i}{\tau}}}\left(\frac{2a^i_i \gamma}{\tau} +\lambda \right. \ldots \\
	&\left. -\sqrt{\del{b_i - \inner{\vec{a}^i,\vec{x}^k} + \del{\frac{2 a^i_i \gamma}{\tau} +\lambda} \sgn(x^k_i)}^2 + 8\lambda \del{\frac{a^i_i}{2} + \frac{a^i_i}{\tau}}|x^k_i|}\right), \\
	&r^{k+1}_i = - r^k_i.
	}\end{align*}}
\end{enumerate}

\end{appendices}


\bibliographystyle{spmpsci}      

\bibliography{dg_refs_smooth_abbrev}

\end{document}